\documentclass[smallextended,referee,envcountsect,]{svjour3}
\smartqed
\usepackage{graphicx}
\journalname{JOTA}

\usepackage{graphicx}

\usepackage{pgfplots}
\usepgfplotslibrary{external} 
\usepackage{multirow}%
\usepackage{amsmath,amssymb,amsfonts}%
\usepackage{mathtools}%
\usepackage{bbold}
\usepackage{mathrsfs}%
\usepackage[title]{appendix}%
\usepackage{xcolor}%
\usepackage{textcomp}%
\usepackage{manyfoot}%
\usepackage{booktabs}%

\usepackage[]{algorithm}%
\usepackage{algorithmic}%

\usepackage{listings}%
\usepackage{enumitem}
\usepackage{parskip}

\usepackage{diagbox}
\usepackage{tabularx}

\usepackage{hyperref}

\renewcommand{\epsilon}{\varepsilon}
\newcommand{\R}{\mathcal{R}}
\newcommand{\tR}{\tilde{\mathcal{R}}}

\newcommand{\Rb}{\mathbb{R}}
\newcommand{\nR}{\nabla \mathcal{R}}

\newcommand{\critic}{\mathcal{C}}
\newcommand{\statio}{\mathcal{E}}

\newcommand{\image}{\R\left(\statio\cap K\right)}

\newtheorem{proofpart}{Part}
\makeatletter
\@addtoreset{proofpart}{theorem}
\makeatother

\begin{document}

\title{Complexities of Armijo-like algorithms in Deep Learning context}

\subtitle{}

\author{Bensaid Bilel}

\institute{Bensaid Bilel \at
             Université de Bordeaux \\
              351, cours de la Libération - F 33 405 TALENCE \\
              bilel.bensaid30@gmail.com
}

\date{Received: date / Accepted: date}

\maketitle

\begin{abstract}
The classical Armijo backtracking algorithm \cite{armijo} achieves the optimal complexity for smooth functions like gradient descent but without any hyperparameter tuning. However, the smoothness assumption is not suitable for Deep Learning optimization. In this work, we show that some variants of the Armijo optimizer achieves acceleration and optimal complexities under assumptions more suited for Deep Learning: the $(L_0,L_1)$ smoothness condition \cite{gd_clipping_accelerates} and analyticity \cite{Bolte_KL}. New dependences on the smoothness constants and the initial gap are established. The results theoretically highlight the powerful efficiency of Armijo-like conditions for highly non-convex problems.   
\end{abstract}


\section{Introduction}

These recent years, machine learning models and more particularly Deep Learning (DL) ones (ex: neural networks) have become state of the art in different fields like computer vision \cite{image_recognition}, language recognition \cite{language_recognition} or physical simulations \cite{plasma,KluthRipoll,LamyFeugeas}. They involve solving highly non-convex optimization problems:
\begin{equation*}
	\displaystyle{\min_{\theta \in \Rb^N}} \R(\theta),
\end{equation*}
where $N \gg 1$ is the dimension of the problem (ex: number of parameters of a neural network) and $\R: \Rb^N \mapsto \Rb$ is a lower bounded differentiable function.
However, the task of minimizing $\R$ is extremely difficult due to the non-convexity of the objective function \cite{DL_opti}.
The research for the global minimum and even the distinction between a local minimum and a saddle point are NP-hard problems
\cite{NP_local_minima}. Moreover, the number of local minimums may increase exponentially with the dimension \cite{expMin}. Therefore, in practice,  practioners look for critical points of $\R$:
\begin{equation*}
	\critic = \{\theta \in \Rb^N, \nR(\theta)=0\}.
	\label{critical_points}
\end{equation*}
When $N$ is big and for complexity reasons, it is commonly accepted to use algorithms of order at most one, in other words, optimizers that need only evaluation of the function and its gradient.

\paragraph{First-order algorithms}
~~\\
A very important number of first-order optimizers have been suggested these recent years to explore the set $\critic$.\\
Schematically, they can be classified in four families:
\begin{itemize}
	\item descent algorithms which move the parameter with a step $\eta>0$ in a research direction $d_n$ such that $d_n \cdot \nR(\theta_n)<0$:
	\begin{equation*}
		\theta_{n+1} = \theta_n - \eta d_n.
	\end{equation*}
	Gradient descent (GD) that is the most famous optimizer of this family corresponds to the case $d_n=-\nR(\theta_n)$. 
	One of the main problem of GD is that one needs the knowledge of the lipschitz constant of the gradient $L$ (smoothness constant) to make it converge. In the litterature, there exists many ways to ajust the learning rate of GD: the Armijo conditions (with the backtracking procedure) \cite{armijo,book_nocedal}, the Wolfe conditions \cite{wolfe1,wolfe2} and the Polyak step size \cite{Boyd_polyak_stepsize,step_polyak_stochastic,sgd_polyak_step_size}. Contrary to the methods presented below, the Armijo/Wolfe methods are not very popular in DL due to the function evaluations needed to compute the step size. However, the results of this paper show that these methods exhibit attractive properties for DL optimization. 
	\item Inertial algorithms suggested by Polyak \cite{Polyak}. The idea consists in using previous gradient evaluations to update the parameters. Contrary to descent algorithms, these methods are optimal for (strongly) convex functions \cite{Polyak,Nesterov}.
	\item Adaptive gradients to adjust automatically the learning rates. In fact, GD is highly sensitive to the choice of the time step $\eta$: it needs an important tuning to get correct performances and to converge. The adaptive gradient methods use the previous gradient evaluations to adjust the time step: see RMSProp \cite{RMSProp}, AdaDelta \cite{Adadelta}, Adam \cite{Adam}, AdaGrad \cite{Adagrad}, RAdam \cite{RAdam} and AMSGrad \cite{AMSGrad}. 
	\item Normalized gradients that aim to reduce the sensitivity to $\eta$ like the previous family. They enforce the time step to be inversely proportional to a power of the gradient: clipping GD/normalized GD \cite{RNN_difficult,gd_clipping_accelerates} and p-GD/rescaled GD \cite{pGD,cv_pflow} are part of this family. They are very popular to train RNN/LSTM network \cite{RNN_difficult} where vanishing gradient and exploding gradient appear \cite{explode1,explode2,explode3}.
\end{itemize}

\paragraph{Classical complexity bounds}
~~\\
The convergence guarantees of all these algorithms have been widely studied under the global $L$-smoothness condition ($\nR$ is $L$-lipschitz continuous on $\Rb^N$) and the Kurdyka-Lojasiewicz inequality \cite{Loj1,Loj2,Kurdyka}. To have an idea of the efficiency of these methods, one common way is to focus on their complexities: given a threshold $\epsilon>0$, we are interested in the number of iterations $n$ needed to satisfy $\|\nR(\theta_n)\| \leq \epsilon$.\\
Let us recall some basic results concerning complexities under classical assumptions in the optimization litterature:
\begin{itemize}
	\item For a convex function, the optimal complexity is given by $\mathcal{O}\left(\epsilon^{-1}\right)$ and is achieved by the Nesterov optimizer \cite{Nesterov_book,Nesterov_amelioration}.
	\item For $L$-smooth function, the optimal complexity is $\mathcal{O}\left(L\epsilon^{-2}\right)$ and GD achieves it \cite{bound_smooth1}.
	\item For a function that admits a $L$-lipschitz gradient and a $\rho$-lipschitz hessian, the optimal complexity is given by $\mathcal{O}\left(\epsilon^{-7/4}\right)$. Many optimizers have been built to achieve this speed \cite{AGD_product_hessian,AGD_guilty_convex,AGD_no_log,parameterfree_AGD,parameterfree_AGD_Holder}.
\end{itemize}
The major limitation of these studies is that the smoothness assumption is not satisfied by DL models \cite{compute_bound_L}. Before presenting some assumptions on the objective function more suited for DL optimization, let us introduce some classical notations in complexity analysis. For an initial point $\theta_0 \in \Rb^N$, let us introduce the following notations:
\begin{equation*}
	\mathcal{S}\coloneqq\{\theta \in \Rb^N, \R(\theta)\leq \R(\theta_0)\},
\end{equation*}
\begin{equation*}
	M\coloneqq\sup_{\theta \in \mathcal{S}} \|\nR(\theta)\|,
\end{equation*}
\begin{equation*}
	\R^* = \inf_{\theta \in \Rb^N} \R(\theta),
\end{equation*}
\begin{equation*}
	\Delta = \R(\theta_0)-\R^*.
\end{equation*}

\paragraph{The generalized smoothness assumption}
~~\\
A relevant hypothesis introduced to analyse deep neural networks and particularly large language models is the $(L_0,L_1)$ smoothness \cite{gd_clipping_accelerates}:
\begin{definition}
	A second order differentiable function $\R$ is $(L_0,L_1)$ smooth if:
	\begin{equation}
		\forall \theta \in \Rb^N, \|\nabla^2 \R(\theta)\| \leq L_0+L_1\|\nR(\theta)\|.
		\label{L0_L1} 
	\end{equation}
\end{definition}
\begin{remark}
	We can assume without loss of generality that $L_1 \neq 0$ since if $L_1=0$, we get back to the classical smoothness assumption.
\end{remark}

This assumption was introduced to explain the superiority of Adam-like algorithms and clipping GD on large language models. In \cite{gd_clipping_accelerates}, this hypothesis is experimentally checked on a family of language tasks. Let us explain why clipping GD is superior to GD under this setting. The authors of \cite{gd_clipping_accelerates} establish a "deceleration" result about GD through the computation of a lower and upper bound for the complexity:
\begin{theorem}[\cite{gd_clipping_accelerates}]
	For the class of functions $(L_0,L_1)$ smooth such that $L_0\geq 1$, $L_1\geq 1$ and $M>1$, the complexity of GD is lower bounded by:
	\begin{equation*}
		\dfrac{L_1M\left(\Delta-5\epsilon/8\right)}{8\epsilon^2\left(\log{M}+1\right)}.
	\end{equation*}
	Let us assume that $M<+\infty$ and let us take $\eta=\dfrac{1}{2(ML_1+L_0)}$. The complexity of GD is upper bounded by:
	\begin{equation*}
		4\left(ML_1+L_0\right)\Delta \epsilon^{-2}.
	\end{equation*}
	\label{GD_L0_L1}
\end{theorem}

On the one hand, this theorem highlights the fact that GD is highly sensitive to the initial condition both through $\Delta$ and $M$. The constant $M$ may be really huge for neural networks \cite{gd_clipping_accelerates,compute_bound_L} explaining the inefficiency of GD to train LSTM. On the other hand, clipping GD does not suffer from this dependence on $M$ (let us denote by $\gamma$ the clipping threshold, see \cite{gd_clipping_accelerates}):
\begin{theorem}[\cite{gd_clipping_accelerates}]
	Let us take $\eta = \frac{1}{10L_0}$ and $\gamma = \min\left(\frac{1}{\eta}, \frac{1}{10L_1\eta}\right)$. The complexity of clipping GD is upper bounded by:
	\begin{equation*}
		\dfrac{20L_0\Delta}{\epsilon^2}+\dfrac{20\max\{1,L_1^2\}\Delta}{L_0}.
	\end{equation*}
	\label{GD_clipping1}
\end{theorem}
It is explained in \cite{gd_clipping_improved} that the lower bound for first order algorithms on this class of functions is of order $\Delta L_0 \epsilon^{-2}$. In \cite{gd_clipping_improved}, the authors prove that clipping GD can almost achieve this speed if $\eta$ and $\gamma$ are well-chosen (same result for Momentum clipping).

\begin{theorem}[\cite{gd_clipping_improved}]
	Let us assume that $\R$ is $(L_0,L_1)$ smooth and let $\epsilon>0$. Let us take $\gamma = \frac{1}{10AL_0}$ and $\eta = \frac{1}{10AL_1}$ where $A=1.06$. Then, we have:
	\begin{equation*}
		\frac{1}{n} \sum_{k=1}^n \|\nR(\theta_k)\| \leq 2\epsilon,
	\end{equation*}
	if $n$ satisfies:
	\begin{equation*}
		n \geq 30\Delta \max\left(\frac{AL_0}{\epsilon^2},25A\frac{L_1^2}{L_0}\right).
	\end{equation*}
	\label{GD_clipping2}
\end{theorem}

Similar results (often less explicit in terms of the constants) are derived for inertial, Adam-like optimizers and sign GD both in the deterministic and the stochastic setting \cite{Adam_generalized_smoothness,adagrad_generalized_smoothness,IG_lipschitz_generalized,adagrad_generalized_smoothness_affine,peter_smoothness}. Some generalizations of $(L_0,L_1)$ smoothness can be find in \cite{signSGD_generalized_smoothness,adam_rho_smooth,l_smoothness}.

\paragraph{Definable and analytic functions}
~~\\
Another class of functions suited for DL optimization includes functions definable on a o-minimal structure and analytic functions, see \cite{Lojasiewicz_gradient,Kurdyka,Bolte_KL,Bolte_semi_analytic}. Although there is a lot of works on this class concerning convergence guarantees \cite{Absil,Lyap_Theory_Bilel,Bilel_ICML,Aujol_recap,HB_nonconvex_acceleration,iPiano,KL_taux,algebraic_momentum,descent_KL}, there is very few works on complexity guarantees for this class of functions \cite{PL_lower_bound,Josz}. In \cite{Josz}, the author proves that for definissable functions on o-minimal structures:
\begin{equation*}
	\min_{0\leq k \leq n-1}\|\nR(\theta_k)\|=o(n^{-1}),
\end{equation*}
compared to the $\mathcal{O}(\sqrt{n})$ for first-order algorithms on smooth functions. However, the author needs the time step to depend both on the lipschitz constant of $\R$ and $\nR$ to derive this bound.

The goal of this work is to analyse the complexity of Armijo-like conditions with a particular backtracking procedure (because they do not need tuning to work), under these new assumptions related to DL optimization:
\begin{itemize}
	\item In section \ref{section_armijo_smooth}, we establish a new and optimal upper bound for the complexity of an Armijo-like algorithm coming from (and implemented in) \cite{Bilel_thesis,Rondepierre}. This bound highlights the same advantage than clipping GD with a better dependence on $L_0,L_1$ constants, without any hyperparameter tuning.
	\item In section \ref{section_armijo_analytic}, we derive an original upper bound on analytic functions for a new Armijo condition drawing inspiration from the interpretation of the classical Armijo inequality suggested in \cite{Lyap_Theory_Bilel}. This bound is optimal on analytic functions with regard to the $\epsilon$-dependence and exhibits a new and interesting relation concerning the gap $\Delta$.
\end{itemize}

\section{Armijo under generalized smoothness assumption}
\label{section_armijo_smooth}

\subsection{The memory backtracking Armijo}

Let us recall first the classical Armijo condition for GD. It consists in finding an adaptive time step $\eta_n$ satisfying:
\begin{equation}
	\R(\theta_n-\eta_n \nR(\theta_n))-\R(\theta_n) \leq -\lambda \eta_n \|\nR(\theta_n)\|^2,
	\label{classical_Armijo}
\end{equation}
where $\lambda \in ]0,1[$ is a constant that could be fixed to $1/2$ in practice \cite{armijo}. For a geometric justification  of this time step constraint, see the book \cite{book_nocedal}. Usually, a backtracking procedure is implemented in order to compute this admissible time step, see \cite{armijo,book_nocedal} for a presentation of the most famous way to implement it. In this work, we study a non-classical backtracking more suitable for DL optimization, presented in Algorithm \ref{algo_LCM}, implemented in \cite{Bilel_thesis} with satisfactory empirical results. In fact, it is claimed and justified in \cite{Rondepierre,Bilel_ICML} and \cite{Bilel_thesis} (chapter2, p55, p58-59) that this procedure needs less function evaluations even for stiff functions. The line 15 of Algorithm \ref{algo_LCM}, called the memory effect, is the main difference with the classical backtracking \cite{book_nocedal}. 

\begin{algorithm}[ht]
	\caption{Memory Armijo}
	\begin{algorithmic}[1]
		\REQUIRE Initial values $\theta_0$, $\eta_{init}$, $\lambda \in ]0,1[$, $f_1>1$, $f_2>1$ and $\epsilon>0$.
		
		\STATE $\theta \leftarrow \theta_0$
		\STATE $g \leftarrow \nR(\theta)$
		
		\WHILE{$\|g\|>\epsilon$}
		\STATE $R_0 \leftarrow \R(\theta)$
		\STATE $\dot{V} \leftarrow \|g\|^2$
		\STATE $\theta_0 \leftarrow \theta$
		\REPEAT
		\STATE $\theta \leftarrow \theta - \eta g$
		\STATE $ R \leftarrow \R(\theta)$
		\IF{$R-R_0>-\lambda \eta \dot{V}$}
		\STATE $\eta \leftarrow \frac{\eta}{f_1}$
		\STATE $\theta \leftarrow \theta_0$
		\ENDIF
		\UNTIL{$R-R_0 \leq -\lambda \eta \dot{V}$}
		\STATE $\eta \leftarrow f_2 \eta$
		\STATE $n \leftarrow n+1$
		\ENDWHILE
		\RETURN $\theta$
	\end{algorithmic}
	\label{algo_LCM}
\end{algorithm}

\subsection{Descent inequality for generalized smoothness}

Let us begin the analysis of Algorithm \ref{algo_LCM} under the $(L_0,L_1)$ smoothness condition. Inequality \eqref{L0_L1} is not easy to handle since the hessian does not appear naturally in the formulation of the algorithm. However, the classical smoothness condition could be formulated in different ways for second order differentiable functions:
\begin{equation}
	\forall \theta \in \Rb^N, \|\nabla ^2 \R(\theta)\| \leq L,
	\label{lip1}
\end{equation}
\begin{equation}
	\forall y_1, y_2 \in \Rb^N, \|\nR(y_2)-\nR(y_1)\| \leq L \|y_2-y_1\|,
	\label{lip2}
\end{equation}
\begin{equation}
	\forall y_1, y_2 \in \Rb^N, \R(y_2)-\R(y_1) \leq \nR(y_1)\cdot (y_2-y_1) + \frac{L}{2}\|y_2-y_1\|^2.
	\label{lip3}
\end{equation}

The formulation of the equation \eqref{lip1} is equivalent to the one of the equation \eqref{L0_L1}. First, we need something that looks like \eqref{lip2} for a $(L_0,L_1)$ smooth function.

\begin{lemma}
	Assume that $\R$ is $(L_0,L_1)$ smooth. Then for $y_1, y_2 \in \Rb^N$, we have:
	\begin{multline*}
		\|\nR(\gamma(t))-\nR(\gamma(0))\| \leq -\frac{L_0}{L_1}-\|\nR(y_1)\|+\left(\frac{L_0}{L_1}+\|\nR(y_1)\|\right) \\
		\exp{\left(L_1\|y_2-y_1\|t\right)},  
	\end{multline*}
	where $\gamma(t) = (1-t) y_1 + t y_2$ for all $t\in [0,1]$.
	\label{lip_general_nR}
\end{lemma}

{\it Proof} 
By Taylor-Lagrange formula, we can write:
\begin{equation*}
	\nR(\gamma(t))-\nR(\gamma(0)) = \int_{0}^t \nabla^2\R(\gamma(\tau))(y_2-y_1)d\tau. 
\end{equation*}
The assumption \eqref{L0_L1} leads to:
\begin{multline*}
	\|\nR(\gamma(t))-\nR(\gamma(0))\| \leq \|y_2-y_1\| \\ \int_{0}^t \left(L_0+L_1\|\nR(\gamma(\tau))-\nR(\gamma(0))+\nR(\gamma(0))\|\right)d\tau.
\end{multline*}
The triangular inequality gives:
\begin{multline*}
	\|\nR(\gamma(t))-\nR(\gamma(0))\| \leq \left(L_0+L_1\|\nR(y_1)\|\right)\|y_2-y_1\|t  \\ + L_1\|y_2-y_1\| \int_{0}^t\|\nR(\gamma(\tau))-\nR(\gamma(0))\|d\tau.
\end{multline*}
By applying Gronwall inequality, we get:
\begin{multline*}
	\|\nR(\gamma(t))-\nR(\gamma(0))\| \leq \left(L_0+L_1\|\nR(y_1)\|\right)\|y_2-y_1\|t \\ + \int_{0}^t L_1\left(L_0+L_1\|\nR(y_1)\right)\|y_2-y_1\|^2 s \exp{\left(L_1\|y_2-y_1\|(t-s)\right)}ds.
\end{multline*}
By an integration by parts of $s\mapsto se^{-s}$, the result is derived. 
\qed

The essential element is to have the equivalent of \eqref{lip3} for $(L_0,L_1)$ smoothness using the previous lemma.

\begin{lemma}
	Assume that $\R$ is $(L_0,L_1)$ smooth. Then for all $y_1, y_2 \in \Rb^N$, we have:
	\begin{multline*}
		\R(y_2)-\R(y_1) \leq \nR(y_1)\cdot (y_2-y_1) - \left(L_0+L_1\|\nR(y_1)\|\right)\|y_2-y_1\| \\ 
		-\frac{1}{L_1^2}\left(L_0+L_1\|\nR(y_1)\|\right) + \frac{1}{L_1^2}\left(L_0+L_1\|\nR(y_1)\|\right)\exp{\left(L_1\|y_2-y_1\|\right)}.
	\end{multline*}
	\label{lip_general_R}
\end{lemma}

{\it Proof} 
The Taylor-Lagrange formula leads to:
	\begin{multline*}
		\R(y_2) = \R(y_1) + \int_{0}^1 \nR(\gamma(t))^T(y_2-y_1)dt \\
		= \R(y_1) + \int_{0}^1 \left[\nR(\gamma(t))-\nR(\gamma(0))+\nR(y_1)\right]^T(y_2-y_1)dt.
	\end{multline*}
	Cauchy-Schwarz inequality gives:
	\begin{equation*}
		\R(y_2) - \R(y_1) \leq \nR(y_1)\cdot (y_2-y_1)+\|y_2-y_1\|\int_0^1\|\nR(\gamma(t))-\nR(\gamma(0))\|dt.
	\end{equation*}
	Applying lemma \ref{lip_general_nR}, we get:
	\begin{multline*}
		\R(y_2) - \R(y_1) \leq \nR(y_1)\cdot (y_2-y_1)\\
		+\|y_2-y_1\|\int_0^1 \left[-\frac{L_0}{L_1}-\|\nR(y_1)\|+\left(\frac{L_0}{L_1}+\|\nR(y_1)\|\right)\exp{\left(L_1\|y_2-y_1\|t\right)} \right]dt.
	\end{multline*}
	Computing this integral finishes the proof.
	\qed
	
\subsection{Complexity result}

For a smooth function, inequality \eqref{lip3} combined with the Armijo rule \eqref{classical_Armijo} leads to the resolution of an affine equation to derive the time step that satisfies uniformly this condition. For $(L_0,L_1)$ smooth function, this leads to the resolution of a Lambert equation. A simple study leads to the following result:
\begin{lemma}
	Let us define the function $h(x) \coloneqq -a -bx + ae^{cx}$ with $a,b,c>0$. Let us assume $b>ac$. If $0 < x \leq \frac{1}{c}\ln{\left(\frac{b}{ac}\right)}$, then $h(x) <0$.
	\label{h_function}
\end{lemma}

Equipped with these tools, we should find a time step depending on $L_0$ and $L_1$ which satisfies the Armijo condition \eqref{classical_Armijo}. According to lemma \ref{lip_general_R}, if $\eta_n$ satisfies \eqref{classical_Armijo}, then the following relation holds:
\begin{multline*}
	-\frac{1}{L_1^2}\left(L_0+L_1\|\nR(\theta_n)\|\right)-\frac{1}{L_1}\left(L_0+L_1(2-\lambda)\|\nR(\theta_n)\|\right)\|\nR(\theta_n)\| \eta_n \\ + \frac{1}{L_1^2}\left(L_0+L_1\|\nR(\theta_n)\|\right)\exp{\left(L_1 \|\nR(\theta_n)\| \eta_n\right)} \leq 0.
\end{multline*}
Using lemma \ref{h_function}, a solution to this inequation is given by:
\begin{equation*}
	\tilde{\eta}_n = \frac{1}{L_1\|\nR(\theta_n)\|}\ln{\left(\dfrac{L_0+L_1(2-\lambda)\|\nR(\theta_n)\|}{L_0+L_1\|\nR(\theta_n)\|}\right)}.
\end{equation*}
At iteration $n$, the time step found by the backtracking satisfies $\eta_n \geq \frac{\tilde{\eta_n}}{f_1}$. Then for all $k \geq 0$, we get:
\begin{multline*}
	\R(\theta_k)-\R(\theta_{k+1}) \geq \lambda \eta_k \|\nR(\theta_k)\|^2 \\ \geq 
	\frac{\lambda}{f_1 L_1}\ln{\left(\dfrac{L_0+L_1(2-\lambda)\|\nR(\theta_k)\|}{L_0+L_1\|\nR(\theta_k)\|}\right)}\|\nR(\theta_k)\|.
\end{multline*}
From now on, let us denote by $n$ the first iteration such that 
$\|\nR(\theta_n)\| \leq \epsilon$. This means that for all $0 \leq k \leq n-1$: $\|\nR(\theta_k)\| > \epsilon$.
Given that the function $z \mapsto \ln{\left(\dfrac{L_0+L_1(2-\lambda)z}{L_0+L_1z}\right)}z$ is increasing on $\Rb_+$,then we have for all $0 \leq k \leq n-1$:
\begin{equation*}
	\R(\theta_k)-\R(\theta_{k+1}) \geq \frac{\lambda}{f_1 L_1}\ln{\left(\dfrac{L_0+L_1(2-\lambda)\epsilon}{L_0+L_1\epsilon}\right)}\epsilon.
\end{equation*}
Summing the inequality above from $k=0$ to $k=n-1$, we get:
\begin{equation*}
	\R(\theta_0)-\R(\theta_n) \geq \frac{\lambda}{f_1 L_1}\ln{\left(\dfrac{L_0+L_1(2-\lambda)\epsilon}{L_0+L_1\epsilon}\right)}\epsilon n.
\end{equation*}
Using the fact that $\R(\theta_n) \geq \R^*$ for all $n\geq 0$ and by inverting the previous inequality, the following bound is derived:
\begin{theorem}
	Let $\lambda \in ]0,1[$, $f_1>1$ and $f_2>1$. Assume tyhat $\R$ is $(L_0,L_1)$ smooth and let $\epsilon>0$. 
	If
	\begin{equation*}
		n\geq f_1 \dfrac{L_1\Delta}{\lambda \epsilon \ln{\left(\frac{L_0+L_1(2-\lambda)\epsilon}{L_0+L_1\epsilon}\right)}},
	\end{equation*}
	then $\displaystyle{\min_{0 \leq k \leq n}}\nR(\theta_k) \leq \epsilon$ where $(\theta_n)_{n \in \mathbb{N}}$ is generated by Algorithm \ref{algo_LCM}.
	\label{LCEGD_eps}
\end{theorem}

Let us comment this theorem.
First of all, this result does not involve $M$, in the same way than clipping GD (form of acceleration according to \cite{gd_clipping_accelerates}). Besides, the given estimation could be simplified for small $\epsilon$ into (for $L_0 \neq 0$):
\begin{equation*}
	f_1\dfrac{L_1\Delta}{\lambda \epsilon \ln{\left(\frac{L_0+L_1(2-\lambda)\epsilon}{L_0+L_1\epsilon}\right)}} \underset{\epsilon \to 0}{\sim} f_1\dfrac{L_0\Delta}{ \lambda (1-\lambda)} \frac{1}{\epsilon^2}.
\end{equation*}
In practice (see \cite{Bilel_thesis,Bilel_ICML}), $f_1=2$ and $\lambda=\frac{1}{2}$ so we have:
\begin{equation}
	f_1 \dfrac{L_1\Delta}{\lambda \epsilon \ln{\left(\frac{L_0+L_1(2-\lambda)\epsilon}{L_0+L_1\epsilon}\right)}} \underset{\epsilon \to 0}{\sim} \dfrac{8 L_0\Delta}{\epsilon^2}.
	\label{LCEGD_eps_equivalent}
\end{equation}
Even in the theorem \ref{GD_clipping2} when $\epsilon \to 0$, $L_1$ is involved in the time step. However algorithm \ref{algo_LCM} is able to achieve the optimal bound (up to a constant) without tuning and dependence on $L_1$ (when $\epsilon \to 0$). It is even possible to improve the factor $8$ by sharpening lemma \ref{h_function}: to do it, we need to look for the optimal solution of the inequation $h(x)\leq 0$ which involves the Lambert function.

A careful reader might argue that the previous estimation deals only with the number of iterations and not the number of gradient evaluations like it is the case for constant step optimizers. In \cite{Bilel_ICML,Bilel_thesis} (chapter2, p55-59), an upper bound on the number of evaluations per iteration is given. It is stated that for neural networks, the number of function evaluations per iteration is equivalent to:
\begin{equation*}
\frac{1}{2} \left[1+\dfrac{\log{(f_2)}}{\log{(f_1)}}\right],
\end{equation*}
gradient evaluations, for algorithm \ref{algo_LCM}.
Therefore, taking into account the backtracking (that is to say the function evaluations), the real number of gradient evaluations for these algorithms, for $f_1=2$, is given by:
\begin{equation*}
	4\left[1+\dfrac{\log{(f_2)}}{\log{2}}\right]\dfrac{L_0\Delta}{\epsilon^2}+\dfrac{8 L_0\Delta}{\epsilon^2}.
\end{equation*}
Compared with theorem \ref{GD_clipping2}, the multiplicative constant in our bound is better for small $\epsilon$ for $f_2\leq 2^{4.5} \approx 23$.

\section{Armijo for analytic functions}
\label{section_armijo_analytic}

In the previous section, we have shown that the classical Armijo condition \eqref{classical_Armijo} with the backtracking procedure of algorithm \ref{algo_LCM} achieves optimal complexity under the generalized smoothness assumption. In this section, we prove an acceleration result on analytic functions. As recalled in the introduction, the Nesterov optimizer achieves the optimal complexity $\mathcal{O}(\epsilon^{-1})$ on smooth convex functions. We get a similar speed for non-convex analytic functions without the need to tune hyperparameters.

\subsection{The explicit-implicit Armijo algorithm}

The new Armijo condition that we are going to state comes from a new interpretation of the classical Armijo condition. In fact, GD is an Euler discretization of the gradient flow:
\begin{equation*}
	\theta'(t) = -\nR(\theta(t)).
\end{equation*}  
The objective function is a Lyapunov function since:
\begin{equation*}
	\frac{d}{dt}\R(\theta(t)) = -\|\nR(\theta(t))\|^2.
\end{equation*}
For all $\lambda \in ]0,1[$, we then have:
\begin{equation}
	\frac{d}{dt}\R(\theta(t)) \leq -\lambda \|\nR(\theta(t))\|^2.
	\label{continuous_armijo}
\end{equation}
In \cite{Lyap_Theory_Bilel}, the authors argue that it is essential to discretize equation \eqref{continuous_armijo} in addition to the gradient flow in order to build a stable method (preservation of the Lyapunov function). Then the inequality \eqref{classical_Armijo} could be seen as an Euler discretization of the left hand side of \ref{continuous_armijo} and an explicit discretization of the right hand side. To have more control on the iterates, we add a semi-implicit discretization of the right hand side:
\begin{equation}
	\R(\theta_{n+1})-\R(\theta_n) \leq -\lambda \eta_n \|\nR(\theta_n)\| \|\nR(\theta_{n+1})\|.
	\label{armijo2}
\end{equation}
Therefore the algorithm consists in finding a time step $\eta_n$ that satisfies both \eqref{classical_Armijo} and \eqref{armijo2} using the backtracking procedure \ref{algo_LCM}. In the same manner as the previous section, using the result of \cite{Bilel_thesis}, we have to multiply the iteration complexity by:
\begin{equation*}
	\frac{3}{2}\left[1+\dfrac{\log{(f_2)}}{\log{(f_1)}}\right]+1
\end{equation*}
to get the number of gradient evaluations (taking into account the function evaluations).
For convenience, let us call this algorithm Explicit-Implicit Armijo (EIA). Let us note that EIA looks like \textbf{IMEX schemes} \cite{IMEX} that combine explicit and implicit discretizations.

\subsection{Convergence of EIA}
Let us first check that EIA converges. Given that EIA satisfies in particular the classical Armijo condition, it is enough to check that the time steps satisfying \eqref{armijo2} are lower bounded (do not vanish asymptotically) according to \cite{Lyap_Theory_Bilel,Rondepierre}. 

From now on, we assume that $\R$ is \textbf{analytic and radially unbounded}. Let us condider the following compact sets:
\begin{equation*}
	K = \overline{K_0+B(0,1)},
\end{equation*}
where:
\begin{equation*}
	K_0 = \overline{\{ \theta_n, n \in \mathbb{N}\}} \subset \mathcal{S}.
\end{equation*}
We denote respectively by $M$ and $L$ the lipschitz constants of $\R$ and $\nR$ on $conv(K)$ ($\nR$ is locally lipschitz continuous). 

\begin{lemma}
	Let us take a time step satisfying:
	\begin{equation*}
		\eta \leq \dfrac{2(1-\lambda)}{1+2\lambda}\dfrac{1}{L}.
	\end{equation*}
	Then for all $\theta \in K$, we have:
	\begin{equation*}
		\R(\theta-\eta \nR(\theta)) - \R(\theta) \leq -\lambda \eta \|\nR(\theta)\|\|\nR(\theta-\eta\nR(\theta))\|.
	\end{equation*}
	\label{lemma_minore_ILC}
\end{lemma}
{\it Proof}
	Let us take $\eta$ satisfying the condition of the lemma.
	By the descent inequality:
	\begin{multline*}
		\R(\theta-\eta \nR(\theta)) - \R(\theta) + \lambda \eta \|\nR(\theta)\|\|\nR(\theta-\eta\nR(\theta))\|  \\ 
		\leq -\eta \|\nR(\theta)\|^2\left(1-\frac{L}{2}\eta\right)+ \lambda \eta \|\nR(\theta)\|\|\nR(\theta-\eta\nR(\theta))\|.
	\end{multline*}
	The triangular inequality in addition with the lipschitzian property of $\nR$ gives:
	\begin{multline*}
		\|\nR(\theta-\eta\nR(\theta))\| = \|\nR(\theta-\eta\nR(\theta))-\nR(\theta)+\nR(\theta)\| \\
		\leq \|\nR(\theta)\| + L\eta \|\nR(\theta)\| \leq (1+L\eta)\|\nR(\theta)\|.
	\end{multline*}
	This leads to:
	\begin{equation*}
		-\|\nR(\theta)\| \leq -\dfrac{1}{1+L\eta}\|\nR(\theta-\eta\nR(\theta))\|.
	\end{equation*}
	Incorporating this in the first inequality and using $\eta\leq \frac{2}{L}$, we get:
	\begin{multline*}
		\R(\theta-\eta \nR(\theta)) - \R(\theta) + \lambda \eta \|\nR(\theta)\|\|\nR(\theta-\eta\nR(\theta))\| \\ 
		\leq -\dfrac{\eta}{1+L\eta}\left(1-\frac{L}{2}\eta\right)\|\nR(\theta)\|\|\nR(\theta-\eta\nR(\theta))\| \\ + \lambda \eta \|\nR(\theta)\|\|\nR(\theta-\eta\nR(\theta))\| \\
		\leq \eta \dfrac{2(\lambda-1)+\eta L(1+2\lambda)}{2+2L\eta}\|\nR(\theta)\|\|\nR(\theta-\eta\nR(\theta))\| \leq 0,
	\end{multline*}
	using the time step condition.
\qed

Combining the previous inequality with proposition 1 of \cite{Lyap_Theory_Bilel} and since $1+2\lambda\geq 1$, the admissible time steps are lower bounded by:
\begin{equation*}
	\forall n \in \mathbb{N}, \eta_n \geq \dfrac{2(1-\lambda)}{f_1(1+2\lambda)}\dfrac{1}{L} \coloneq \eta^*.
\end{equation*}
Using theorem 6 of \cite{Lyap_Theory_Bilel}, the sequence $(\theta_n)_{n\in \mathbb{N}}$ generated by EIA converges to a critical point of $\R$.

\subsection{Global Lojasiewicz inequality}
In order to estimate precisely the complexity of EIA, we need a global Lojasiewicz inequality to bound accurately the sum of increments $\|\theta_{n+1}-\theta_n\|=\eta_n \|\nR(\theta_n)\|$ for two reasons:
\begin{itemize}
	\item first the proof of the $\mathcal{O}(\epsilon^{-2})$ bound for smooth function goes through the estimation of the sum of $\eta_n \|\nR(\theta_n)\|^2$  \cite{Nesterov_book} so it is natural to focus on $\eta_n \|\nR(\theta_n)\|$ to derive a $\mathcal{O}(\epsilon^{-1})$ bound.
	\item In the proof of the convergence of $(\theta_n)$ \cite{Lyap_Theory_Bilel}, the local Lojasiewicz inequality enables only to estimate the rest of the serie $\|\theta_{n+1}-\theta_n\|$ since we need $n$ large enough to apply the Lojasiewicz inequality in the neighborhood of a critical point. 
\end{itemize}

\begin{lemma}
	There exists an increasing and concave function  $\phi: [0,+\infty[ \mapsto [0,+\infty[$ differentiable on $]0,+\infty[$ satisfying:
	\begin{equation}
		\theta \in K, \R(\theta) \notin \image \implies \phi'\left(\tR(\theta)\right)\|\nR(\theta)\|\geq 1,
	\end{equation}
	where $\tR(\theta) = d\left(\R(\theta),\image\right)$. 
	\label{lemma_phi}
\end{lemma}

{\it Proof}
According to the Morse-Sard theorem about analytic functions \cite{morse_sard_analytic}, $\image$ is finite:
\begin{equation*}
	\exists s \in \mathbb{N}, \image = \{\R_1, \dots, \R_s\}.
\end{equation*}
We arrange these values in increasing order: $\R_1 < \R_2 < \dots < \R_s$.
Using corollary 13 in \cite{KL_uniform}, it exists $\gamma>0$ and a function $\phi: [0,\gamma[ \mapsto [0,+\infty[$ of the form $c\dfrac{x^{\alpha}}{\alpha}$ for $0<\alpha\leq \frac{1}{2}$ and $c>0$ satisfying:
\begin{equation}
	\forall \theta \in K, \forall i \in \{1,\dots, s\}, 0<|\R(\theta)-\R_i|<\gamma \implies \|\nabla \left(\phi \circ |\R-\R_i|\right)(\theta)\| \geq 1.
	\label{local_phi}
\end{equation}
Let us define the following quantities:
\begin{equation*}
	\delta_0 \coloneq \min_{1\leq j\neq i \leq s} \dfrac{|\R_j-\R_i|}{2}>0,
\end{equation*}
\begin{equation*}
	0<\delta < \min(\delta_0,\gamma),
\end{equation*}
\begin{equation*}
	a \coloneq \inf\left\{\|\nR(\theta)\|, \theta \in K \text{ tel que } \forall i=1,\dots,s; |\R(\theta)-\R_i|\geq \delta\right\}.
	\label{def_a}
\end{equation*}
Note that $a>0$ since it is the lower bound of a continuous function that does not vanish, on a compact set. Then, we distinguish two cases depending on the value of $\phi'(\delta)$:
\begin{itemize}
	\item if $\phi'(\delta)\geq \frac{1}{a}$ then we extend $\phi$ on $[\delta,+\infty[$ by an affine function with a slope given by $\phi'(\delta)$. In this case, the function $\phi$ is defined in the following way:
	\begin{equation*}
		\phi(x)\coloneq
		\left\{
		\begin{array}{ll}
			c\frac{x^{\alpha}}{\alpha} \text{ if } x\leq \delta,\\
			c\delta^{\alpha-1}(x-\delta)+c\frac{\delta^{\alpha}}{\alpha} \text{ else}.
		\end{array}
		\right.
	\end{equation*}
	\item Else, we normalize $\phi$ by $a \phi'(\delta)$ then we extend it in the same way. As a result, $\phi$ is defined by:
	\begin{equation*}
		\phi(x)\coloneq
		\left\{
		\begin{array}{ll}
			c\dfrac{x^{\alpha}}{\alpha ca\delta^{\alpha-1}} \text{ if } x\leq \delta,\\\\
			\dfrac{c\delta^{\alpha-1}}{ca\delta^{\alpha-1}}(x-\delta)+c\dfrac{\delta^{\alpha}}{\alpha ca \delta^{\alpha-1}} \text{ else}.
		\end{array}
		\right.
	\end{equation*}
\end{itemize}
We sum up the two cases by the existence of constants $C_1,C_2,C_3>0$
such that the function $\phi$ is upper bounded by:
\begin{equation}
	\phi(x)\leq
	\left\{
	\begin{array}{ll}
		C_1 x^{\alpha} \text{ if } x\leq \delta,\\\\
		C_2(x-\delta)+C_3 \text{ else}.
	\end{array}
	\right.
	\label{majoration_phi}
\end{equation}
~~\\
Now, let us check that $\phi$ satisfies the property claimed. Let $\theta \in K$ such that $\R(\theta) \notin \image$. We consider two cases:
\begin{enumerate}
	\item There is $i \in \{1,\dots,s\}$ such that $|\R(\theta)-\R_i|\leq \delta$. As $\delta<\delta_0$, for all $\tilde{\theta}$ in a neighborhood of $\theta$, we have:
	\begin{equation*}
		|\R(\tilde{\theta})-\R_i| = \tR(\tilde{\theta}).
	\end{equation*}
	By the chain rule and by applying \eqref{local_phi} ($\delta<\gamma$), we get:
	\begin{equation*}
		\phi'(\tR(\theta)) \|\nR(\theta)\|\geq 1.
	\end{equation*}
	\item The second case corresponds to:
	\begin{equation*}
		\forall i \in \{1,\dots,s\}, |\R(\theta)-\R_i|>\delta.
	\end{equation*}
	By definition of $a$, the following immediately holds:
	\begin{equation*}
		\phi'(\tR(\theta)) \|\nR(\theta)\| \geq \phi'(\tR(\theta))a \geq 1.
	\end{equation*}
\end{enumerate}
Therefore, the inequality is checked in all the cases.
\qed

Let us recall that a concave positive function is sub-additive:
\begin{equation*}
	\forall a,b \in \Rb_+, \phi(a+b) \leq \phi(a)+\phi(b).
\end{equation*}
Then for all $y\geq x\geq 0$, the sub-additivity could be applied with $y=x+(y-x)$ to write:
\begin{equation}
	\phi(y)-\phi(x)\leq \phi(y-x).
	\label{phi_contractant}
\end{equation}
We will need this property later in the proof.

\begin{remark}
	Note that, in the bound \eqref{majoration_phi}, the constants $C_1,C_2,C_3$ depend only on $\alpha$, $c$ (lowest eigenvalue of the hessian if $\R$ is a Morse function), $\delta$ (size of attractive regions) and possibly on $a$ that is related to the value of the amplitude of the gradient at the edge of the basins of attraction.
\end{remark}

\subsection{Closed intervals without critical values}
Here, we establish an inequality about the length of the iterates, that is to say the sum $\displaystyle{\sum_{k=p}^{q-1}}\|\theta_{k+1}-\theta_k\|$ when 
$\left[\R(\theta_q),\R(\theta_p)\right]$ does not include any critical values ($\R_i$) and successive mean of critical values ($(\R_i+\R_{i+1})/2$). \textbf{For convenience, we call this condition the assumption (a)}. In a nutshell, between two successive critical values $\R_i$ and $\R_{i+1}$, we only focus on the closed intervals included in $\left](\R_{i}+\R_{i+1})/2, \R_{i+1}\right[$ or $\left]\R_i,(\R_{i}+\R_{i+1})/2\right[$.

\begin{lemma}
	Let $p\leq q$ two integers such that the interval $\left[\R(\theta_q),\R(\theta_p)\right]$ satisfies assumption (a). Then the following inequality holds:
	\begin{equation}
		\sum_{k=p}^{q-1}\|\theta_{k+1}-\theta_k\| \leq \frac{1}{\lambda} \phi\left(\R(\theta_p)-\R(\theta_q)\right).
	\end{equation} 
	\label{lemma_sum_phi1}
\end{lemma}

{\it Proof}
Under the assumption (a), there are only two possibilities:
\begin{itemize}
	\item case 1: $\forall k\in\{p,\dots,q-1\}$, $\tR(\theta_{k+1}) \geq \tR(\theta_k)$. In such a case, we are restricted to an interval of the form $\left](\R_i+\R_{i+1})/2,\R_{i+1}\right[$ in which all the $\tR(\theta_k)$ are arranged in a monotone order. This enables to make a telescopic sum appear.  
	\item Case 2: $\forall k\in\{p,\dots,q-1\}$, $\tR(\theta_{k+1}) \leq \tR(\theta_k)$. In such a case, we are restricted to an interval of the form $\left]\R_i,(\R_i+\R_{i+1})/2\right[$.
\end{itemize}
Let us deal with the first case. Let $k\in\{p,\dots,q-1\}$. We multiply the inequality \eqref{armijo2} by $\phi'\left(\tR(\theta_{k+1})\right)>0$, which gives using lemma \ref{lemma_phi}:
\begin{multline*}
	\|\theta_{k+1}-\theta_k\| \leq \|\theta_{k+1}-\theta_k\| \left[\phi'(\tR(\theta_{k+1}))\|\nR(\theta_{k+1})\|\right] \\ \leq \dfrac{\phi'\left(\tR(\theta_{k+1})\right)}{\lambda} \left[\R(\theta_k)-\R(\theta_{k+1})\right]. 
\end{multline*}
Under assumption (a), we get:
\begin{equation}
	\exists i \in \{1,\dots,s\}, \forall k \in\{p,\dots,q-1\}, \tR(\theta_k) = \R_i-\R(\theta_k).
	\label{dist_exist}
\end{equation}
Therefore injecting this equality in the previous inequality, it leads to:
\begin{equation*}
	\|\theta_{k+1}-\theta_k\| \leq \dfrac{\phi'\left(\tR(\theta_{k+1})\right)}{\lambda} \left[\tR(\theta_{k+1})-\tR(\theta_k)\right].
\end{equation*}
By concavity of $\phi$ (tangents below chords), we have:
\begin{equation*}
	\phi'\left(\tR(\theta_{k+1})\right) \left[\tR(\theta_{k+1})-\tR(\theta_k)\right] \leq \phi\left(\tR(\theta_{k+1})\right)-\phi\left(\tR(\theta_k)\right).
\end{equation*}
Therefore, the following holds:
\begin{equation*}
	\|\theta_{k+1}-\theta_k\| \leq \frac{1}{\lambda} \phi\left(\tR(\theta_{k+1})\right)-\phi\left(\tR(\theta_k)\right).
\end{equation*}
By summing the previous inequality from $k=p$ to $q-1$, we are allowed to write:
\begin{equation}
	\sum_{k=p}^{q-1}\|\theta_{k+1}-\theta_k\| \leq \phi\left(\tR(\theta_q)\right)-\phi\left(\tR(\theta_p)\right).
	\label{sum_theta1}
\end{equation}
Using the property \eqref{phi_contractant}, we get:
\begin{equation*}
	\phi\left(\tR(\theta_q)\right)-\phi\left(\tR(\theta_p)\right) \leq \phi\left(\tR(\theta_q)-\tR(\theta_p)\right) = \phi\left(\R(\theta_p)-\R(\theta_q)\right).
\end{equation*}
where the last equality comes from \eqref{dist_exist}.
By combining the previous inequality with the equation \eqref{sum_theta1}, we obtain the majoration of the lemma.
For the second case, we follow exactly the same approach with the following differences: 
\begin{itemize}
	\item the first inequality is obtained by multiplying \eqref{classical_Armijo} by $\phi'\left(\tR(\theta_k)\right)>0$.
	\item Assumption (a) implies the following property:
	\begin{equation*}
		\exists i \in \{1,\dots,s\}, \forall k \in\{p,\dots,q-1\}, \tR(\theta_k) = \R(\theta_k)-\R_i.
	\end{equation*}
\end{itemize}
We then obtain the claimed result in the second case.
\qed

\subsection{Semi-open intervals without critical values}
Here, we deal with the case where $\left[\R(\theta_q),\R(\theta_p)\right[$ satisfies assumption (a) in the event that $\R(\theta_p)$ is a critical value or the mean of two successive critical values. 

\begin{lemma}
	Let $\epsilon>0$ ($\epsilon$ of the stopping criterion) and $d<1$. Let $p\leq q$ be integers such that the interval $\left[\R(\theta_q),\R(\theta_p)\right[$ satisfies assumption (a) in the event that $\R(\theta_p)$ is a critical value or the mean of two successive critical values. Let us assume $\displaystyle{\min_{p \leq k \leq q-1}} \|\nR(\theta_k)\|>\epsilon$. Then the following inequality holds:
	\begin{equation}
		\sum_{k=p}^{q-1}\|\theta_{k+1}-\theta_k\| \leq \frac{1}{d \lambda} \phi\left(\R(\theta_p)-\R(\theta_q)\right).
	\end{equation} 
	\label{lemma_sum_phi2}
\end{lemma}
{\it Proof}
This proof is composed of four parts.
In the first part, we are going to build a sequence of points which approximate each of the iterates $(\theta_n)_{p\leq n \leq q-1}$. In the next two parts, we are going to show that this sequence satisfies the two Armijo inequalities \eqref{classical_Armijo} and \eqref{armijo2} up to a factor $d\lambda$ in place of $\lambda$. This enables to come back to the case of lemma \ref{lemma_sum_phi1}. Passing to the limit, we deduce the aforementioned result in the last part.
\begin{proofpart}
	If $\theta_p$ is a local minimum then $\nR(\theta_p)=0$. As a result, the sum of $\|\theta_{k+1}-\theta_k\|$ is zero and the inequality of the lemma is automatically satisfied. Otherwise, there is a sequence of initial conditions $(\theta_p^l)_{l \in \mathbb{N}}$ such that:
	\begin{equation*}
		\forall l \in \mathbb{N}, \R(\theta_p^l) < \R(\theta_p),
	\end{equation*}
	and
	\begin{equation*}
		\lim_{l\to +\infty} \theta_p^l = \theta_p.
	\end{equation*}
	Given that to get assumption (a), it is sufficient to exclude a finite number of elements, there is $\tilde{\delta}>0$ such that $\left[\R(\theta_q)+\tilde{\delta},\R(\theta_p)\right[$ satisfies hypothesis (a). At fixed $l$, we are going to build a finite sequence of points $\left(\theta_n^l\right)_{p \leq n \leq q-1}$ with initial condition $\theta_p^l$ using the time steps $(\eta_n)_{p \leq n \leq q-1}$, generated by EIA with $\theta_p$ as the initial condition:
	\begin{equation*}
		\forall l\in \mathbb{N}, \forall n \in \{p, \dots, q-1\}, \theta_{n+1}^l = \theta_n^l -\eta_n \nR(\theta_n^l),
	\end{equation*}
	where $\eta_n$ satisfies:
	\begin{equation}
		\R(\theta_{n+1})-\R(\theta_n) \leq -\lambda \eta_n \|\nR(\theta_n)\|^2,
		\label{armijo10}
	\end{equation}
	and
	\begin{equation}
		\R(\theta_{n+1})-\R(\theta_n) \leq -\lambda \eta_n \|\nR(\theta_n)\| \|\nR(\theta_{n+1})\|.
		\label{armijo20}
	\end{equation}
	As $\eta_n$ does not depend on $l$, we have immediately by induction that for all $p \leq n \leq q-1$:
	\begin{equation*}
		\lim_{l \to +\infty} \theta_n^l = \theta_n.
	\end{equation*}
	However, seeing $\eta_n$ does not depend on $l$, it is not obvious that the Armijo conditions \eqref{armijo10} and \eqref{armijo20} are satisfied for $\theta_n^l$ when $l$ is large enough. The goal of the next two parts is to prove such a result, which makes it possible to derive an inequality as the same kind as lemma\ref{lemma_sum_phi1} for $l$ large.\\ 
	Thereafter, let us introduce the following notations:
	\begin{equation*}
		\eta_{min} \coloneq \min_{p \leq k \leq q-1} \eta_k \text{ and } \eta_{max}\coloneq \max_{p \leq k \leq q-1} \eta_k.
	\end{equation*}
	The time steps $\eta_{min}$ and $\eta_{max}$ are strictly positive since the time steps are lower bounded by a strictly positive constant by lemma \ref{lemma_minore_ILC}.
\end{proofpart}

\begin{proofpart}
	Let us define the following constants (they may seem arbitrary for the moment):
	\begin{equation*}
		r_0^0 \coloneq \dfrac{1}{1+\eta_{max}L},
	\end{equation*}
	\begin{equation*}
		r_0^1 \coloneq \dfrac{1}{M\left[2+\eta_{max}(L+2d\lambda)\right]},
	\end{equation*}
	\begin{equation*}
		r_0^2 \coloneq \dfrac{1}{\sqrt{L}\sqrt{3/2+d\lambda \eta_{max}L+\eta_{max}^2L^2}},
	\end{equation*}
	\begin{equation*}
		r_0 \coloneq \min\left(\dfrac{1-d}{2}\lambda \eta_{min}\epsilon^2 \min\left(r_0^0, r_0^1, r_0^2\right),1\right).
	\end{equation*}
	Let $\theta \in \{\theta_p,\dots,\theta_{q-1}\}$ and $\eta$ its associated time step (that is to say $\eta=\eta_k$ if $\theta=\theta_k$ for $p\leq k \leq q-1$). Let $h \in \Rb^N$ such that $\|h\| \leq r_0$. Let us focus on the following quantity:
	\begin{equation*}
		f(\theta,h) = \R(\theta+h-\eta\nR(\theta+h))-\R(\theta+h)+d\lambda \eta \|\nR(\theta+h)\|^2.
	\end{equation*}
	If we prove $f(\theta,h) \leq 0$, this will mean that the points at a distance $r_0$ of $\theta$ satisfies \eqref{classical_Armijo} up to a factor $\lambda d$ (in place of $\lambda$) with the time step $\eta$ (corresponding to $\theta$). \\
	First, as $\theta \in K_0$ and $r_0\leq 1$, $\theta+h \in K$. Moreover, we get:
	\begin{multline*}
		\|\theta+h-\eta \nR(\theta+h)\| \leq \|\theta-\eta \nR(\theta)-\eta \left[\nR(\theta+h)-\nR(\theta)\right]+h\| \\ \leq \|\theta-\eta \nR(\theta)\| + \left[1+\eta L\right]\|h\|.
	\end{multline*}
	As $\|h\| \leq r_0^0$ and $\theta-\eta \nR(\theta) \in K_0$ then $\theta+h-\eta \nR(\theta+h) \in K$. We will upper bound $f(\theta,h)$. We split up $f(\theta,h)$ into:
	\begin{multline*}
		f(\theta,h) = \R\left(\theta-\eta\nR(\theta)+h+\eta\left[\nR(\theta)-\nR(\theta+h)\right]\right) - \R(\theta+h) \\ + d\lambda \eta \|\nR(\theta)+\left[\nR(\theta+h)-\nR(\theta)\right]\|^2.
	\end{multline*}
	We use the local smoothness property to obtain:
	\begin{multline}
		\forall x,y\in conv(K), \R(x)+\nR(x)\cdot(y-x)-\frac{L}{2}\|y-x\|^2 \\ \leq \R(y) \leq \R(x)+\nR(x)\cdot(y-x)+\frac{L}{2}\|y-x\|^2.
		\label{L_regularity}
	\end{multline}
	This leads to:
	\begin{multline*}
		f(\theta,h) \leq \R(\theta-\eta \nR(\theta))+\nR(\theta-\eta \nR(\theta))\cdot \left[h+\eta \left[\nR(\theta)-\nR(\theta+h)\right]\right]\\ +\frac{L}{2}\|h+\eta \left[\nR(\theta)-\nR(\theta+h)\right]\|^2\\ 
		-\R(\theta)-\nR(\theta)\cdot h+\frac{L}{2}\|h\|^2+d\lambda \eta \|\nR(\theta)\|^2 \\
		+d\lambda \eta \|\nR(\theta+h)-\nR(\theta)\|^2+2d\lambda \eta \|\nR(\theta)\|\|\nR(\theta+h)-\nR(\theta)\|.
	\end{multline*}
	Using Cauchy-Scharwz inequality on the two scalar products and the fact that for $\theta \in K$, $\|\nR(\theta)\|\leq M$, $\|\nR(\theta+h)-\nR(\theta)\|\leq L\|h\|$, as well as $\|x+y\|^2 \leq 2\left(\|x\|^2+\|y\|^2\right)$, we get:
	\begin{multline*}
		f(\theta,h) \leq \left[\R(\theta-\eta \nR(\theta))-\R(\theta)+\lambda \eta \|\nR(\theta)\|^2\right] \\ + (d-1)\lambda \eta \|\nR(\theta)\|^2+M\left[\|h\|+\eta L \|h\|\right]\\
		+L\left[\|h\|^2+\eta^2L^2\|h\|^2\right]
		+M\|h\|+\frac{L}{2}\|h\|^2+d\lambda \eta L^2\|h\|^2+2d\lambda \eta ML\|h\|.
	\end{multline*}
	Since the term in the first square bracket is negative (see \eqref{armijo10}), we have:
	\begin{multline*}
		f(\theta,h) \leq (d-1)\lambda \eta \|\nR(\theta)\|^2+\|h\|M\left[2 +\eta_{max}L(1+2d\lambda)\right]\\ +\|h\|^2L\left[3/2+d\lambda \eta_{max}L+\eta_{max}^2L^2\right].
	\end{multline*}
	Since $\|h\| \leq r_0$, the following upper bound holds:
	\begin{multline*}
		f(\theta,h) \leq (d-1)\lambda \eta \|\nR(\theta)\|^2+(1-d)\lambda \eta_{min}\epsilon^2 \\ = (1-d) \left(\eta_{min}\epsilon^2-\eta\|\nR(\theta)\|^2\right) \leq 0.
	\end{multline*}
\end{proofpart}
\begin{proofpart}
	The next step is to prove a similar inequality for the second Armijo condition \eqref{armijo2}. It is very similar to part 2.
	Let us define the following radii:
	\begin{equation*}
		r_1^0 \coloneq \dfrac{1}{1+\eta_{max}L},
	\end{equation*}
	\begin{equation*}
		r_1^1 \coloneq \dfrac{1}{M\left[2+(1+2d\lambda)\eta_{max}L+d\lambda \eta_{max}L^2\right]},
	\end{equation*}
	\begin{equation*}
		r_1^2 \coloneq \dfrac{1}{\sqrt{L}\sqrt{3/2+d\lambda \eta_{max}L+(1+d\lambda)\eta_{max}^2L^2}},
	\end{equation*}
	\begin{equation*}
		r_1 = \coloneq \min\left(\dfrac{1-d}{2}\lambda \eta_{min}\epsilon^2 \min\left(r_1^0, r_1^1, r_1^2\right),1\right).
	\end{equation*}
	Let $h\in \Rb^N$ such that $\|h\|\leq r_1$. Let us focus on the following quantity:
	\begin{multline*}
		g(\theta,h) \coloneq \R(\theta+h-\eta\nR(\theta+h))-\R(\theta+h) \\
		+d\lambda \eta \|\nR(\theta+h)\|\|\nR\left(\theta+h-\eta\nR(\theta+h)\right)\|.
	\end{multline*}
	we rewrite $g(\theta,h)$ in the following way:
	\begin{multline*}
		g(\theta,h) = \R\left(\theta-\eta\nR(\theta)+h+\eta\left[\nR(\theta)-\nR(\theta+h)\right]\right) - \R(\theta+h) \\
		+d\lambda \eta \|\nR(\theta+h)\| \|\nR\left(\theta-\eta\nR(\theta)+h+\eta\left[\nR(\theta)-\nR(\theta+h)\right]\right)\|. 
	\end{multline*}
	Using \eqref{L_regularity}, we get:
	\begin{multline*}
		g(\theta,h) \leq \R(\theta-\eta \nR(\theta))+\nR(\theta-\eta \nR(\theta))\cdot \left[h+\eta \left[\nR(\theta)-\nR(\theta+h)\right]\right]\\ +\frac{L}{2}\|h+\eta \left[\nR(\theta)-\nR(\theta+h)\right]\|^2\\ 
		-\R(\theta)-\nR(\theta)\cdot h+\frac{L}{2}\|h\|^2 + 
		d\lambda \eta \|\nR(\theta+h)-\nR(\theta)+\nR(\theta)\| \\ \times \|\nR\left(\theta-\eta\nR(\theta)+h+\eta\left[\nR(\theta)-\nR(\theta+h)\right]\right)-\nR(\theta-\eta\nR(\theta))+\nR(\theta-\eta \nR(\theta))\|.
	\end{multline*}
	Proceeding in the same way than the previous part, we obtain:
	\begin{multline*}
		g(\theta,h) \leq \R(\theta-\eta\nR(\theta))-\R(\theta)+M\|h\|+ML\eta \|h\|+L\|h\|^2+L^3\eta^2\|h\|^2 \\
		+M\|h\|+\frac{L}{2}\|h\|^2+d\lambda \eta \left[\|\nR(\theta)\|+L\|h\|\right]\left[\|\nR\left(\theta-\eta\nR(\theta)\right)\| + L\|h-\eta \left[\nR(\theta+h)-\nR(\theta)\right]\| \right].
	\end{multline*}
	Rearranging the terms and since $\eta \leq \eta_{max}$, we have:
	\begin{multline*}
		g(\theta,h) \leq \left[\R(\theta-\eta\nR(\theta))-\R(\theta)+\lambda \eta \|\nR(\theta)\|\|\nR\left(\theta-\eta\nR(\theta)\right)\|\right] \\ 
		+(d-1)\lambda \eta \|\nR(\theta)\| \|\nR\left(\theta-\eta\nR(\theta)\right)\| \\ + M \left[2+(1+2d\lambda)\eta_{max}L+d\lambda \eta_{max}L^2\right]\|h\| \\
		+ L\left[3/2+d\lambda \eta_{max}L+(1+d\lambda)\eta_{max}^2L^2\right]\|h\|^2.
	\end{multline*}
	Since the term in first square bracket is negative by \eqref{armijo20} and $\|h\|\leq r_1$, it comes:
	\begin{equation*}
		g(\theta,h) \leq (d-1)\lambda \left[\eta \|\nR(\theta)\| \|\nR\left(\theta-\eta \nR(\theta)\right)\| - \eta_{min}\epsilon^2\right] \leq 0.
	\end{equation*}
	Combining the last previous parts, we get for all $n \in \{p,\dots,q-1\}$ , if $\theta \in B(\theta_n,\min(r_0,r_1))$ then:
	\begin{equation}
		\begin{array}{ll}
			\R(\theta-\eta_n\nR(\theta))-\R(\theta)\leq -d\lambda \eta_n \|\nR(\theta)\|^2, \\
			\R(\theta-\eta_n\nR(\theta))-\R(\theta) \leq -d\lambda \eta_n \|\nR(\theta)\|\|\nR\left(\theta-\eta_n\nR(\theta)\right)\|.
		\end{array}
		\label{dissipation_perturb}
	\end{equation}
\end{proofpart}
\begin{proofpart}
	Now, we should find a neighbordhood of $\theta_p$ in such a way that all the points generated from an initial condition in this neighborhood are at a distance at most $\min(r_0,r_1)$ from the points $\theta_n$ for $p\leq n\leq q-1$. Let us define the following radius:
	\begin{equation*}
		r \coloneq \dfrac{\min(r_0,r_1)}{(1+\eta_{max}L)^{q-p}}.
	\end{equation*}
	Consider $l_0 \in \mathbb{N}$ satisfying:
	\begin{equation*}
		\forall l \geq l_0, \theta_p^l \in B(\theta_p,r).
	\end{equation*}
	From now on, $l \geq l_0$. We will show by induction on $n \in \{p,\dots,q\}$ the following property:
	\begin{equation}
		\theta_n^l \in B\left(\theta_n,\dfrac{\min(r_0,r_1)}{(1+\eta_{max}L)^{q-n}}\right).
	\end{equation}
	The initialisation is obvious by definition of $l_0$. Let us assume that the property is true at rank $n \in \{p+1,\dots,q-1\}$. Then we write:
	\begin{multline*}
		\|\theta_{n+1}^l-\theta_{n+1}\| = \|\theta_n^l-\eta_n\nR(\theta_n^l)-\left[\theta_n-\eta_n\nR(\theta_n)\right]\| \\ = \|\theta_n^l-\theta_n-\eta_n \left[\nR(\theta_n^l)-\nR(\theta_n)\right]\|.
	\end{multline*}
	The traingular inequality gives:
	\begin{equation*}
		\|\theta_{n+1}^l-\theta_{n+1}\| \leq \|\theta_n^l-\theta_n\|+\eta_{max}L\|\theta_n^l-\theta_n\| \leq (1+\eta_{max}L)\|\theta_n^l-\theta_n\|.
	\end{equation*}
	Using the induction hypothesis, we have:
	\begin{equation*}
		\|\theta_{n+1}^l-\theta_{n+1}\| \leq \dfrac{\min(r_0,r_1)}{\left(1+\eta_{max}L\right)^{q-(n+1)}}.
	\end{equation*}
	Therefore, the property is inherited. \\
	Since for all $n \in \{p,\dots,q-1\}$, $\dfrac{\min(r_0,r_1)}{(1+\eta_{max}L)^{q-n}} \leq \min(r_0,r_1)$, the following Armijo conditions hold:
	\begin{equation*}
		\begin{array}{ll}
			\R(\theta_n^l-\eta_n\nR(\theta_n^l))-\R(\theta_n^l)\leq -d\lambda \eta_n \|\nR(\theta_n^l)\|^2, \\
			\R(\theta_n^l-\eta_n\nR(\theta_n^l))-\R(\theta_n^l) \leq -d\lambda \eta_n \|\nR(\theta_n^l)\|\|\nR\left(\theta_n^l-\eta_n\nR(\theta_n^l)\right)\|.
		\end{array}
	\end{equation*}
	For $l$ large enough, we also have for all $n\in \{p,\dots,q-1\}$, $\R(\theta_n^l) \in [\R(\theta_q)+\tilde{\delta},\R(\theta_p)[$. Let us recall that $[\R(\theta_q)+\tilde{\delta},\R(\theta_p)[$ also satisfies assumption (a) in the event that $\R(\theta_p)$ is a critical value or the mean of two successive critical values. Let us also recall that $\R(\theta_p^l)$ is not a critical value even if $\R(\theta_p)$ is, since $\R(\theta_p^l)<\R(\theta_p)$. \\
	By applying lemma \ref{lemma_sum_phi1} to the interval $[\R(\theta_q^l),\R(\theta_p^l)]$, it comes for $l$ large enough:
	\begin{equation*}
		\sum_{k=p}^{q-1}\|\theta_{k+1}^l-\theta_k^l\| \leq \frac{1}{d \lambda} \phi\left(\R(\theta_p^l)-\R(\theta_q^l)\right).
	\end{equation*}
	Passing to the limit when $l\to +\infty$, we get the expected result. 
\end{proofpart}
\qed

In the lemmas \ref{lemma_sum_phi1} and \ref{lemma_sum_phi2}, the sum does not take into account the $q$-th term. Indeed, it might be that $[\R(\theta_{q+1}),\R(\theta_q))]$ does not satisfy assumption (a). In this case, we upper bound $\|\theta_{q+1}-\theta_q\|$ using the classical Armijo condition \eqref{classical_Armijo}:
\begin{equation*}
	\|\theta_{q+1}-\theta_q\|^2 \leq \dfrac{\R(\theta_q)-\R(\theta_{q+1})}{\lambda \eta_q}.
\end{equation*}
Since $\R(\theta_q)\leq \R(\theta_0)$ and $\R(\theta_{q+1})\geq \R^*$, we deduce:
\begin{equation*}
	\|\theta_{q+1}-\theta_q\| \leq \sqrt{\dfrac{\Delta}{\lambda \eta_q}}.
\end{equation*}

As a result, in the setups of lemmas \ref{lemma_sum_phi1} and \ref{lemma_sum_phi2}, the following holds for $d<1$:
\begin{equation}
	\sum_{k=p}^{q}\|\theta_{k+1}-\theta_k\| \leq \frac{1}{d \lambda} \phi\left(\R(\theta_p)-\R(\theta_q)\right)+\sqrt{\dfrac{\Delta}{\lambda \eta_q}}.
	\label{inegalite_troncon}
\end{equation} 

\subsection{Results}
From the previous lemmas, we will upper bound the length of the sequence $(\theta_n)_{n\in \mathbb{N}}$. This result will be useful in order to measure the complexity of EIA. 

\begin{theorem}
	Let $\epsilon>0$ and $n\geq 0$ satisfying $\displaystyle{\min_{0\leq k\leq n-1}} \|\nR(\theta_k)\|>\epsilon$. Then it exists $\tilde{s}\leq 2s-1$ such that for all $d<1$, we have:
	\begin{equation*}
		\sum_{k=0}^{n-1}\|\theta_{k+1}-\theta_k\| \leq \frac{\tilde{s}+1}{d\lambda} \phi\left(\dfrac{\R(\theta_0)-\R(\theta_n)}{\tilde{s}+1}\right)+(\tilde{s}+1)\sqrt{\dfrac{\Delta}{\lambda \displaystyle{\min_{0\leq k \leq n-1}}\eta_k}}, 
	\end{equation*}
	where $s$ is the number of critical values of $\R$ on $K$. \\
	More specifically, there are constants $c,\delta>0$ and $0<\alpha\leq \frac{1}{2}$ as well as multiplicative constants $C_1,C_2,C_3>0$ such that:
	\begin{equation*}
		\sum_{k=0}^{n-1}\|\theta_{k+1}-\theta_k\| \leq \frac{\tilde{s}+1}{d\lambda} 
		\left\{
		\begin{array}{ll}
			C_1 \left(\dfrac{\Delta}{2s}\right)^{\alpha} \text{ if } \dfrac{\Delta}{2s}\leq \delta \\
			C_2\left(\dfrac{\Delta}{2s}-\delta\right)+C_3\text{ else }
		\end{array}
		\right.
		+(\tilde{s}+1)\sqrt{\dfrac{\Delta}{\lambda \displaystyle{\min_{0\leq k \leq n-1}}\eta_k}}.
	\end{equation*}
	\label{th_longeur}
\end{theorem}
{\it Proof}
	We split up $[\R(\theta_n),\R(\theta_0)[$ into $\tilde{s}$ pieces satisfying each the assumptions of lemma \ref{lemma_sum_phi2}:
	\begin{equation*}
		[\R(\theta_n),\R(\theta_0)[ = [\R(\theta_{n_{\tilde{s}+1}}),\R(\theta_{n_{\tilde{s}}})[ \cup \dots [\R(\theta_{n_2}),\R(\theta_{n_1})[ \cup [\R(\theta_{n_1}),\R(\theta_{n_0})[,
	\end{equation*}
	where $n_0\coloneq 0$ and $n_{\tilde{s}+1}\coloneq n$ for convenience. Note that $\tilde{s}\leq 2s-1$ since we locate the values of the iterates in intervals of the form $]\R_i,(\R_i+\R_{i+1})/2[$ or $](\R_i+\R_{i+1})/2,\R_{i+1}[$ for $i\in \{1,\dots,s-1\}$.  Let $d<1$. Let us split up the length of the sequence of iterates to apply inequality \eqref{inegalite_troncon}:
	\begin{equation*}
		\sum_{k=0}^{n-1} \|\theta_{k+1}-\theta_k\| = \sum_{j=0}^{\tilde{s}} \sum_{k=n_j}^{n_{j+1}} \|\theta_{k+1}-\theta_k\| \leq \sum_{j=0}^{\tilde{s}}\left[ \frac{1}{d\lambda}\phi\left(\R(\theta_{n_j})-\R(\theta_{n_{j+1}})\right) + \sqrt{\dfrac{\Delta}{\lambda \eta_{n_j}}}\right].
	\end{equation*}
	By using Jensen inequality on $\phi$ and by lower bounding the time steps (for the second term), we get that the previous sum is upper bounded by:
	\begin{equation*}
		\dfrac{\tilde{s}+1}{d\lambda}\phi\left(\frac{1}{\tilde{s}+1}\sum_{j=0}^{\tilde{s}}\R(\theta_{n_j})-\R(\theta_{n_{j+1}})\right)+(\tilde{s}+1)\sqrt{\dfrac{\Delta}{\lambda \displaystyle{\min_{0\leq k \leq n-1}}\eta_k}}.
	\end{equation*}
	The first result arises from the computation of the telescopic sum appearing in $\phi$. The second one comes from the bound $\eqref{majoration_phi}$ and the fact that $\R(\theta_0)-\R(\theta_n) \leq \Delta$ ($\phi$ is increasing).
\qed

\begin{remark}
	Let us note that in theorem \ref{th_longeur}, $s$ could be taken as the number of critical values in $K_0 \subset \left\{\theta \in \Rb^N, \R(\theta) \leq \R(\theta_0)\right\}$. Indeed, the addition of the ball $B(0,1)$ to $K_0$ is arbitrary and we could add a ball of radius arbitrary small in order to not add critical values to $K$ compared to $K_0$. \\
	The result has been established for analytic functions but in fact two structures are sufficient to come up with such a bound: a local Kurdyka-Lojasiewicz inequality and a finite number of critical values on all compact set.
\end{remark}

\begin{corollary}
	Let $\epsilon>0$ and $d<1$. If $n$ satisfies:
	\begin{equation*}
		n \geq \dfrac{1}{\epsilon} \left[\dfrac{\tilde{s}+1}{d\lambda\eta^*}\phi\left(\dfrac{\Delta}{\tilde{s}+1}\right)+(\tilde{s}+1)\sqrt{\dfrac{\Delta}{\lambda}}(\eta^*)^{-3/2}\right],
	\end{equation*}
	then we have $\displaystyle{\min_{0\leq k \leq n}} \|\nR(\theta_k)\| \leq \epsilon$.
	\label{corollaire_acceleration_epsilon}
\end{corollary}
{\it Proof}
	Let us denote by $n$ the first iteration such that $\|\nR(\theta_n)\|\leq \epsilon$. Then for all $0\leq k \leq n-1$, $\|\nR(\theta_k)\|> \epsilon$. Since $\|\theta_{k+1}-\theta_k\| = \eta_k \|\nR(\theta_k)\|$, the theorem \ref{th_longeur} gives:
	\begin{equation*}
		\epsilon \sum_{k=0}^{n-1}\eta_k \leq \frac{\tilde{s}+1}{d\lambda} \phi\left(\dfrac{\Delta}{\tilde{s}+1}\right)+(\tilde{s}+1)\sqrt{\dfrac{\Delta}{\lambda \displaystyle{\min_{0\leq k \leq n-1}}\eta_k}}.
	\end{equation*}
	Given that for all $k\in \mathbb{N}$, $\eta_k \geq \eta^*$, we get:
	\begin{equation*}
		n \leq \dfrac{1}{\epsilon} \left[\dfrac{\tilde{s}+1}{d\lambda\eta^*}\phi\left(\dfrac{\Delta}{\tilde{s}+1}\right)+(\tilde{s}+1)\sqrt{\dfrac{\Delta}{\lambda}}(\eta^*)^{-3/2}\right].
	\end{equation*}
	The result is then proved.
\qed

\begin{remark}
	The previous corollary could be seen as an acceleration result given that we achieve a $\mathcal{O}(\epsilon^{-1})$ complexity contrary to the $\mathcal{O}(\epsilon^{-2})$ bound of GD on smooth functions. However, the term $(\eta^*)^{-3/2}$ may involve $L^{3/2}$ in the worst case: this is due to the rough way we deal with the transition terms, that is to say the $\R(\theta_{n_j})$. \\
	The corollary makes a second noteworthy improvement concerning the $\Delta$-dependence. If $\frac{\Delta}{\tilde{s}+1} \leq \delta$ then the dependence of the complexity on $\Delta$ involves only $\Delta^{\alpha}$ and $\sqrt{\Delta}$. For a Morse function, $\Delta^{\alpha}=\sqrt{\Delta}$. Let us recall that $\delta$ measures both the size of the attractive regions and the gap between critical values. At the best of our knowledge, it is the first bifurcation phenomenon on $\Delta$ in the optimization litterature. \\
	For the first term (involving $\phi$), a bifurcation on $\tilde{s}$ appears since the dependence on $\tilde{s}$ is not linear anymore but it behaves like $\tilde{s}^{1-\alpha}$ if $\dfrac{\Delta}{\tilde{s}+1}\leq \delta$. However, the second term depends linearly on $\tilde{s}$: so it is relevant to better estimate the transition terms in future works.
\end{remark}

\begin{remark}
	Nesterov \cite{Nesterov_book} proves that the optimal complexity for first-order algorithms on smooth convex functions is around $\mathcal{O}(\epsilon^{-1})$. The function representing the worst case is the following (called Nesterov's worst function) for $N$ an odd number:
	\begin{equation*}
		\R(\theta) = \theta_1^2 + \sum_{k=1}^{N-1} \left[\theta_i-\theta_{i+1}\right]^2+\theta_N^2,
	\end{equation*}
	where $\theta_i$ denotes the $i$-th component of the vector  $\theta$. This function is analytic and also makes it possible to derive the optimal complexity for analytic functions. Therefore the dependence on $\epsilon^{-1}$ is optimal among analytic functions and it achieves by EIA.  
\end{remark}

\begin{corollary}
	The iterates of the EIA algorithm satisfy:
	\begin{equation*}
		\min_{0\leq k \leq n-1}\|\nR(\theta_k)\|=o(n^{-1}).
	\end{equation*}
	\label{corollaire_petito}
\end{corollary}

\begin{proof}
	Let us define for $n\in \mathbb{N}$, $g_n = \displaystyle{\min_{0\leq k \leq n-1}\|\nR(\theta_k)\|}$. Since $(g_n)_{n\in \mathbb{N}}$ is a positive decreasing sequence it is clear that:
	\begin{equation*}
		(n/2+1)g_n \leq \left(n- \lfloor n/2 \rfloor +1\right)g_n \leq \sum_{\lfloor n/2 \rfloor}^n g_k \leq \sum_{\lfloor n/2 \rfloor}^{+\infty} g_k.
	\end{equation*}
	Therefore we have:
	\begin{equation*}
		g_n \leq \dfrac{2}{n+2}\sum_{k=\lfloor n/2 \rfloor}^{+\infty}g_k.
	\end{equation*}
	To get the result, it is sufficient that the remainder tends to $0$ when $n\to \infty$. According to theorem \ref{th_longeur}, we get:
	\begin{equation*}
		\eta^* \sum_{k=0}^{n-1}g_k \leq \frac{\tilde{s}+1}{d\lambda} \phi\left(\dfrac{\Delta}{\tilde{s}+1}\right)+(\tilde{s}+1)\sqrt{\dfrac{\Delta}{\lambda \eta^*}}.
	\end{equation*}
	Then the sum of $(g_k)_{k\in \mathbb{N}}$ converges and then the remainder tends to $0$. 
\end{proof}

\begin{remark}
	This result can be put in touch with the one of \cite{Nesterov_amelioration} where it is shown that the algorithm Nesterov \it Forward-Backward optimizer satisfies: $\|\theta_{n+1}-\theta_n\|=o(n^{-1})$ and $\R(\theta_n)-\R^* = o(n^{-2})$. \\
	As recalled in the introduction it can also be in touch with proposition 8 of \cite{Josz} but EIA does not need any tuning.  
\end{remark}

\section{Conclusion}

To conclude, let us sum up the main contributions of the paper:
\begin{itemize}
	\item first we prove that the memory backtracking Armijo (implemented in \cite{Bilel_thesis} on machine learning tasks) achieves the same acceleration than clipping GD compared to GD. In addition, it achieves optimal complexity under the generalized smoothness assumption and improves the complexity of clipping GD even in terms of gradient/function evaluations.
	\item We establish an acceleration complexity of $\mathcal{O}(\epsilon^{-1})$ for analytic functions compared to the $\mathcal{O}(\epsilon^{-2})$ for smooth functions. The algorithm EIA achieves the speed $o(n^{-1})$ of \cite{Nesterov_amelioration,Josz} but without any hyperparameter tuning. A very important point is the bifurcation on $\Delta$ that involves the improvement in $\sqrt{\Delta}$, that depends on the number of critical values and the size of the basins of attraction. 
\end{itemize}
These results show theoretically that it will be interesting to explore the applicability of Armijo-like algorithms in Deep Learning. 

\section{Declarations}

\paragraph{Funding}
No funding was received for conducting this study.

\paragraph{Competing interests}
The authors have no competing interests to declare that are relevant to the content of this article.

\newpage

\bibliographystyle{spmpsci}
\bibliography{biblio}

\begin{thebibliography}{10}
\providecommand{\url}[1]{{#1}}
\providecommand{\urlprefix}{URL }
\expandafter\ifx\csname urlstyle\endcsname\relax
  \providecommand{\doi}[1]{DOI~\discretionary{}{}{}#1}\else
  \providecommand{\doi}{DOI~\discretionary{}{}{}\begingroup
  \urlstyle{rm}\Url}\fi

\bibitem{sgd_polyak_step_size}
{S}tochastic {P}olyak {S}tep-{S}ize for {SGD}: {A}n {A}daptive {L}earning
  {R}ate for {F}ast {C}onvergence (2020)

\bibitem{Absil}
Absil, P., Mahony, R., Andrews, B.: {Convergence of the Iterates of Descent
  Methods for Analytic Cost Functions}.
\newblock SIAM Journal on Optimization \textbf{16}(2), 531--547 (2005).
\newblock \doi{10.1137/040605266}.
\newblock \urlprefix\url{https://doi.org/10.1137/040605266}

\bibitem{HB_nonconvex_acceleration}
Apidopoulos, V., Ginatta, N., Villa, S.: {Convergence Rates for the Heavy-Ball
  Continuous Dynamics for Non-Convex Optimization, under Polyak–Łojasiewicz
  Condition} \textbf{84}, 563–589.
\newblock \doi{https://doi.org/10.1007/s10898-022-01164-w}

\bibitem{armijo}
Armijo, L.: {Minimization of Functions Having Lipschitz Continuous First
  Partial Derivatives.}
\newblock Pacific Journal of Mathematics \textbf{16}(1), 1 -- 3 (1966)

\bibitem{Bolte_KL}
Attouch, H., Bolte, J., Redont, P., Soubeyran, A.: {Proximal Alternating
  Minimization and Projection Methods for Nonconvex Problems: An Approach Based
  on the Kurdyka-Łojasiewicz Inequality}.
\newblock Mathematics of Operations Research \textbf{35}(2), 438--457 (2010).
\newblock \doi{10.1287/moor.1100.0449}.
\newblock \urlprefix\url{https://doi.org/10.1287/moor.1100.0449}

\bibitem{descent_KL}
Attouch, H., Bolte, J., Svaiter, B.: {C}onvergence of {D}escent {M}ethods for
  {S}emi-{A}lgebraic and {T}ame {P}roblems: {P}roximal {A}lgorithms,
  {F}orward-{B}ackward {S}plitting, and {R}egularized {G}auss-{S}eidel
  {M}ethods.
\newblock Mathematical Programming \textbf{137}, 36 (2011).
\newblock \doi{10.1007/s10107-011-0484-9}

\bibitem{Nesterov_amelioration}
Attouch, H., Peypouquet, J.: {The Rate of Convergence of Nesterov\'s
  Accelerated Forward-Backward Method is actually $o(k^{-2})$}.
\newblock SIAM Journal on Optimization \textbf{26} (2015).
\newblock \doi{10.1137/15M1046095}

\bibitem{expMin}
Auer, P., Herbster, M., Warmuth, M.: {Exponentially Many Local Minima for
  Single Neurons}.
\newblock In: Advances in Neural Information Processing Systems, vol.~8. MIT
  Press (1996).
\newblock
  \urlprefix\url{https://proceedings.neurips.cc/paper/1995/file/3806734b256c27e41ec2c6bffa26d9e7-Paper.pdf}

\bibitem{Aujol_recap}
Aujol, J., Dossal, C., Rondepierre, A.: {C}onvergence {R}ates of the
  {H}eavy-{B}all {M}ethod under the {L}ojasiewicz {P}roperty.
\newblock Mathematical Programming \textbf{198}(1), 195–254 (2022).
\newblock \doi{10.1007/s10107-022-01770-2}.
\newblock \urlprefix\url{https://doi.org/10.1007/s10107-022-01770-2}

\bibitem{Bilel_thesis}
Bensaid, B.: {Analyse et D{\'e}veloppement de Nouveaux Optimiseurs en Machine
  Learning}.
\newblock Theses, {Universit{\'e} de Bordeaux} (2024).
\newblock \urlprefix\url{https://theses.hal.science/tel-04808839}

\bibitem{Lyap_Theory_Bilel}
Bensaid, B., Poette, G., Turpault, R.: {An Abstract Lyapunov Control Optimizer:
  Local Stabilization and Global Convergence}.
\newblock Arxiv  (2024)

\bibitem{Bilel_ICML}
Bensaid, B., Poette, G., Turpault, R.: {Convergence of the Iterates for
  Momentum and RMSProp for Local Lipshitz Functions: Adaptation is the Key}.
\newblock In: Arxiv (2024)

\bibitem{Bolte_semi_analytic}
Bolte, J., Daniilidis, A., Lewis, A.: {The Łojasiewicz Inequality for
  Nonsmooth Subanalytic Functions with Applications to Subgradient Dynamical
  Systems}.
\newblock Society for Industrial and Applied Mathematics \textbf{17},
  1205--1223 (2007).
\newblock \doi{10.1137/050644641}

\bibitem{KL_uniform}
Bolte, J., Daniilidis, A., Lewis, A., Shiota, M.: Clarke subgradients of
  stratifiable functions.
\newblock SIAM Journal on Optimization \textbf{18}(2), 556--572 (2007).
\newblock \doi{10.1137/060670080}.
\newblock \urlprefix\url{https://doi.org/10.1137/060670080}

\bibitem{Boyd_polyak_stepsize}
Boyd, S., Mutapcic, A.: {Subgradient Methods}.
\newblock lecture notes of EE392o, Stanford University, Autumn Quarter
  \textbf{2004} (2003)

\bibitem{bound_smooth1}
Carmon, Y., Duchi, J., Hinder, O., al.: {Lower Bounds for Finding Stationary
  Points I} \textbf{184}, 71--120 (2020).
\newblock \doi{https://doi.org/10.1007/s10107-019-01406-y}

\bibitem{AGD_guilty_convex}
Carmon, Y., Duchi, J., Hinder, O., Sidford, A.: {"Convex Until Proven Guilty":
  Dimension-Free Acceleration of Gradient Descent on Non-Convex Functions}.
\newblock In: International Conference on Machine Learning (ICML) (2017)

\bibitem{AGD_product_hessian}
Carmon, Y., Duchi, J., Hinder, O., Sidford, A.: {Accelerated Methods for
  NonConvex Optimization}.
\newblock SIAM Journal on Optimization \textbf{28}(2), 1751--1772 (2018).
\newblock \doi{10.1137/17M1114296}

\bibitem{signSGD_generalized_smoothness}
Crawshaw, M., Liu, M., Orabona, F., Zhang, W., Zhuang, Z.: {Robustness to
  Unbounded Smoothness of Generalized SignSGD}.
\newblock In: Neural Information Processing Systems (NeurIPS) (2022)

\bibitem{peter_smoothness}
E.Gorbunov, N.Tupitsa, S.Choudhury, Aliev, A., P.Richtárik, S.Horváth,
  M.Takáč: Methods for convex \$(l\_0,l\_1)\$-smooth optimization: Clipping,
  acceleration, and adaptivity.
\newblock In: International Conference on Learning Representations (2024).
\newblock \urlprefix\url{https://openreview.net/forum?id=0wmfzWPAFu}

\bibitem{adagrad_generalized_smoothness_affine}
Faw, M., Root, L., Caramanis, C., Shakkottai, S.: {Beyond Uniform Smoothness: A
  Stopped Analysis of Adaptive SGD}.
\newblock In: Annual Conference on Learning Theory, vol. 195, pp. 1--72 (2023)

\bibitem{explode3}
Gehring, J., Auli, M., Grangier, D., Yarats, D., Dauphin, Y.: {Convolutional
  Sequence to Sequence Learning}.
\newblock In: Proceedings of the 34th International Conference on Machine
  Learning, \emph{ICML'17}, vol.~70, p. 1243–1252. JMLR.org (2017)

\bibitem{compute_bound_L}
Herrera, C., Krach, F., Teichmann, J.: {L}ocal {L}ipschitz {B}ounds of {D}eep
  {N}eural {N}etworks (2023)

\bibitem{adam_rho_smooth}
H.Li, A.Rakhlin, A.Jadbabaie: Convergence of adam under relaxed assumptions.
\newblock In: International Conference on Neural Information Processing
  Systems, NIPS '23. Curran Associates Inc., Red Hook, NY, USA (2024)

\bibitem{l_smoothness}
H.Li, J.Qian, Y.Tian, A.Rakhlin, A.Jadbabaie: Convex and non-convex
  optimization under generalized smoothness.
\newblock In: International Conference on Neural Information Processing
  Systems, NIPS '23. Curran Associates Inc., Red Hook, NY, USA (2024)

\bibitem{gd_clipping_accelerates}
{J. Zhang and T. He and S. Sra and A. Jadbabaie}: {Why Gradient Clipping
  Accelerates Training: A Theoretical Justification for Adaptivity}.
\newblock In: International Conference on Learning Representations (2020).
\newblock \urlprefix\url{https://openreview.net/forum?id=BJgnXpVYwS}

\bibitem{KL_taux}
Jia, Z., Wu, Z., Dong, X.: {A}n {I}nexact {P}roximal {G}radient {A}lgorithm
  with {E}xtrapolation for a {C}lass of {N}onconvex {N}onsmooth {O}ptimization
  {P}roblems.
\newblock Journal of Inequalities and Applications \textbf{2019}, 16 (2019).
\newblock \doi{10.1186/s13660-019-2078-7}

\bibitem{Josz}
Josz, C.: Global convergence of the gradient method for functions definable in
  o-minimal structures.
\newblock Mathematical Programming \textbf{202}, 385 (2023).
\newblock \doi{10.1007/s10107-023-01972-2}.
\newblock \urlprefix\url{https://doi.org/10.1007/s10107-023-01972-2}

\bibitem{algebraic_momentum}
Josz, C., Lai, L., Li, X.: {C}onvergence of the {M}omentum {M}ethod for
  {S}emialgebraic {F}unctions with {L}ocally {L}ipschitz {G}radients.
\newblock SIAM Journal on Optimization \textbf{33}(4), 3012--3037 (2023).
\newblock \doi{10.1137/23M1545720}.
\newblock \urlprefix\url{https://doi.org/10.1137/23M1545720}

\bibitem{Adam}
Kingma, D., Ba, J.: {Adam: A Method for Stochastic Optimization}.
\newblock In: International Conference on Learning Representations (ICLR)
  (2014).
\newblock \doi{10.48550/ARXIV.1412.6980}

\bibitem{plasma}
Kluth, G., Humbird, K., Spears, B., Peterson, J., Scott, H., Patel, M., Koning,
  J., Marinak, M., Divol, L., Young, C.: Deep {L}earning for {NLTE} {S}pectral
  {O}pacities.
\newblock Physics of Plasmas \textbf{27}(5), 052,707 (2020).
\newblock \doi{10.1063/5.0006784}

\bibitem{Kurdyka}
Kurdyka, K.: {On Gradients of Functions Definable in o-Minimal Structures}.
\newblock Annales de l'Institut Fourier \textbf{48}(3), 769--783 (1998).
\newblock \doi{10.5802/aif.1638}.
\newblock
  \urlprefix\url{https://aif.centre-mersenne.org/articles/10.5802/aif.1638/}

\bibitem{LamyFeugeas}
Lamy, C., Dubroca, B., Nicola\"{\i}, P., Tikhonchuk, V., Feugeas, J.: Modeling
  of {E}lectron {N}onlocal {T}ransport in {P}lasmas using {A}rtificial {N}eural
  {N}etworks.
\newblock Phys. Rev. E \textbf{105}, 055,201 (2022).
\newblock \doi{10.1103/PhysRevE.105.055201}.
\newblock \urlprefix\url{https://link.aps.org/doi/10.1103/PhysRevE.105.055201}

\bibitem{AGD_no_log}
Li, H., Lin, Z.: {Restarted Nonconvex Accelerated Gradient Descent: No More
  Polylogarithmic Factor in the $O(\epsilon^{-7/4})$ Complexity}.
\newblock In: Proceedings of the 39th International Conference on Machine
  Learning, \emph{Proceedings of Machine Learning Research}, vol. 162, pp.
  12,901--12,916. PMLR (2022).
\newblock \urlprefix\url{https://proceedings.mlr.press/v162/li22o.html}

\bibitem{Adam_generalized_smoothness}
Li, H., Rakhlin, A., Jadbabaie, A.: {Convergence of Adam Under Relaxed
  Assumptions}.
\newblock In: Thirty-Seventh Conference on Neural Information Processing
  Systems (2023).
\newblock \urlprefix\url{https://openreview.net/forum?id=yEewbkBNzi}

\bibitem{RAdam}
Liu, L., Jiang, H., He, P., Chen, W., Liu, X., Gao, J., Han, J.: {On the
  Variance of the Adaptive Learning Rate and Beyond}.
\newblock p.~13. ICLR (2020).
\newblock \doi{10.48550/ARXIV.1908.03265}

\bibitem{step_polyak_stochastic}
Loizou, N., Vaswani, S., Laradji, I., Lacoste-Julien, S.: {Stochastic Polyak
  Step-Size for SGD: An Adaptive Learning Rate for Fast Convergence}.
\newblock In: International Conference on Artificial Intelligence and
  Statistics, pp. 1306--1314. PMLR (2021)

\bibitem{Loj1}
{\L}ojasiewicz, S.: {A Topological Property of Real Analytic Subsets} (1963)

\bibitem{Lojasiewicz_gradient}
{\L}ojasiewicz, S.: {Trajectories of the Gradient of an Analytic Function}.
\newblock Seminari di Geometria. Universit{\'a} degli Studi di Bologna
  \textbf{1982/1983}, 115--117 (1984)

\bibitem{Loj2}
{\L}ojasiewicz, S.: {On Semi- and Subanalytic Geometry}.
\newblock Annales de l'Institut Fourier \textbf{43}(5), 1575--1595 (1993).
\newblock \doi{https://doi.org/10.1007/s10915-023-02215-4}

\bibitem{Adagrad}
Lydia, A., Francis, S.: {Adagrad - An Optimizer for Stochastic Gradient
  Descent}.
\newblock International journal of information and computing science
  \textbf{Volume 6}, 566--568 (2019)

\bibitem{parameterfree_AGD}
Marumo, N., Takeda, A.: {Parameter-Free Accelerated Gradient Descent for
  Nonconvex Minimization} (2024)

\bibitem{parameterfree_AGD_Holder}
Marumo, N., Takeda, A.: {Universal Heavy-Ball Method for Nonconvex Optimization
  under H\"older Continuous Hessians} (2024)

\bibitem{explode1}
Merity, S., Keskar, N., Socher, R.: {Regularizing and Optimizing {LSTM}
  Language Models}

\bibitem{image_recognition}
Mo, W., Luo, X., Zhong, Y., Jiang, W.: {Image Recognition Using Convolutional
  Neural Network Combined with Ensemble Learning Algorithm}.
\newblock Journal of Physics: Conference Series \textbf{1237}, 022,026 (2019).
\newblock \doi{10.1088/1742-6596/1237/2/022026}

\bibitem{NP_local_minima}
Murty, K., Kabadi, S.: {Some NP-Complete Problems in Quadratic and Nonlinear
  Programming} \textbf{39}, 117--129 (1987).
\newblock \doi{https://doi.org/10.1007/BF02592948}

\bibitem{Nesterov}
Nesterov, Y.: {A method for Solving the Convex Programming Problem with
  Convergence Rate $O(1/k^2)$}.
\newblock Proceedings of the USSR Academy of Sciences \textbf{269}, 543--547
  (1983)

\bibitem{Nesterov_book}
Nesterov, Y.: {Introductory Lectures on Convex Optimization} (2003).
\newblock \doi{https://doi.org/10.1007/978-1-4419-8853-9}

\bibitem{book_nocedal}
Nocedal, J., Wright, S.: {Numerical Optimization}.
\newblock Springer New York (2006).
\newblock \doi{https://doi.org/10.1007/978-0-387-40065-5}

\bibitem{Rondepierre}
Noll, D., Rondepierre, A.: {Convergence of Linesearch and Trust-Region Methods
  using the Kurdyka-Łojasiewicz Inequality}.
\newblock Springer Proceedings in Mathematics and Statistics \textbf{50}
  (2013).
\newblock \doi{10.1007/978-1-4614-7621-427}

\bibitem{iPiano}
Ochs, P.: {L}ocal {C}onvergence of the {H}eavy-{B}all {M}ethod and ipiano for
  {N}on-{C}onvex {O}ptimization.
\newblock Journal of Optimization Theory and Applications \textbf{177}(1),
  153--180 (2018).
\newblock \doi{https://doi.org/10.1007/s10957-018-1272-y}

\bibitem{RNN_difficult}
Pascanu, R., Mikolov, T., Bengio, Y.: {On the difficulty of training recurrent
  neural networks}.
\newblock In: International Conference on International Conference on Machine
  Learning (ICML) (2013)

\bibitem{explode2}
Peters, M., Neumann, M., Iyyer, M., Gardner, M., Christopher, C., al.: {Deep
  Contextualized Word Representations}.
\newblock In: Proceedings of the 2018 Conference of the North {A}merican
  Chapter of the Association for Computational Linguistics: Human Language
  Technologies, pp. 2227--2237. Association for Computational Linguistics
  (2018).
\newblock \doi{10.18653/v1/N18-1202}.
\newblock \urlprefix\url{https://aclanthology.org/N18-1202}

\bibitem{Polyak}
Polyak, B.: {Some Methods of Speeding Up the Convergence of Iteration Methods}.
\newblock USSR Computational Mathematics and Mathematical Physics
  \textbf{4}(5), 1--17 (1964).
\newblock \doi{https://doi.org/10.1016/0041-5553(64)90137-5}

\bibitem{IG_lipschitz_generalized}
Qian, J., Wu, Y., Zhuang, B., Wang, S., Xiao, J.: {Understanding Gradient
  Clipping In Incremental Gradient Methods}.
\newblock In: A.~Banerjee, K.~Fukumizu (eds.) Proceedings of The 24th
  International Conference on Artificial Intelligence and Statistics,
  \emph{Proceedings of Machine Learning Research}, vol. 130, pp. 1504--1512.
  PMLR (2021).
\newblock \urlprefix\url{https://proceedings.mlr.press/v130/qian21a.html}

\bibitem{AMSGrad}
Reddi, S., Kale, S., Kumar, S.: {On the Convergence of Adam and Beyond}.
\newblock In: International Conference on Learning Representations (ICLR)
  (2018).
\newblock \doi{10.48550/ARXIV.1904.09237}

\bibitem{language_recognition}
Richardson, F., Reynolds, D., Dehak, N.: {Deep Neural Network Approaches to
  Speaker and Language Recognition}.
\newblock IEEE Signal Processing Letters \textbf{22}(10), 1671--1675 (2015).
\newblock \doi{10.1109/LSP.2015.2420092}

\bibitem{KluthRipoll}
Ripoll, J., Kluth, G., Has, S., Fischer, A., Mougeot, M., Camporeale, E.:
  Exploring {P}itch-{A}ngle {D}iffusion during {H}igh {S}peed {S}treams with
  {N}eural {N}etworks.
\newblock In: 2022 3rd URSI Atlantic and Asia Pacific Radio Science Meeting
  (AT-AP-RASC), pp. 1--4 (2022).
\newblock \doi{10.23919/AT-AP-RASC54737.2022.9814235}

\bibitem{cv_pflow}
Romero, O., Benosman, M., Pappas, G.: {ODE discretization schemes as
  optimization algorithms}.
\newblock In: Conference on Decision and Control (CDC) (2022).
\newblock \doi{10.1109/CDC51059.2022.9992691}

\bibitem{morse_sard_analytic}
Souček, J., Souček, V.: Morse-sard theorem for real-analytic functions.
\newblock Commentationes Mathematicae Universitatis Carolinae \textbf{013}(1),
  45--51 (1972).
\newblock \urlprefix\url{http://eudml.org/doc/16473}

\bibitem{DL_opti}
Sun, R.: Optimization for {D}eep {L}earning: {T}heory and {A}lgorithms (2019)

\bibitem{RMSProp}
Tieleman, T., Hinton, G., al: {Lecture 6.5-rmsprop: Divide the Gradient by a
  Running Average of its Recent Magnitude}.
\newblock COURSERA: Neural networks for machine learning \textbf{4}(2), 26--31
  (2012)

\bibitem{adagrad_generalized_smoothness}
Wang, B., Zhang, H., Ma, Z., Chen, W.: {Convergence of AdaGrad for Non-convex
  Objectives: Simple Proofs and Relaxed Assumptions}.
\newblock In: Annual Conference on Learning Theory, vol. 195, pp. 1--30 (2023)

\bibitem{pGD}
Wilson, A., Mackey, L., Wibisono, A.: {Accelerating Rescaled Gradient Descent:
  Fast Optimization of Smooth Functions}.
\newblock In: Neural Information Processing Systems (NeurIPS) (2019)

\bibitem{wolfe1}
Wolfe, P.: {Convergence Conditions for Ascent Methods}.
\newblock SIAM Review \textbf{11}(2), 226--235 (1969).
\newblock \doi{10.1137/1011036}.
\newblock \urlprefix\url{https://doi.org/10.1137/1011036}

\bibitem{wolfe2}
Wolfe, P.: {Convergence Conditions for Ascent Methods. II: Some Corrections}.
\newblock SIAM Review \textbf{13}(2), 185--188 (1971).
\newblock \doi{10.1137/1013035}.
\newblock \urlprefix\url{https://doi.org/10.1137/1013035}

\bibitem{PL_lower_bound}
Yue, P., Fang, C., Lin, Z.: {On the Lower Bound of Minimizing
  Polyak-Łojasiewicz Functions}.
\newblock In: G.~Neu, L.~Rosasco (eds.) Annual Conference on Learning Theory
  (2023)

\bibitem{Adadelta}
Zeiler, M.: {ADADELTA: An adaptive learning rate method} \textbf{1212} (2012).
\newblock \doi{10.48550/ARXIV.1212.5701}

\bibitem{gd_clipping_improved}
Zhang, B., Jin, J., Fang, C., Wang, L.: {Improved Analysis of Clipping
  Algorithms for Non-convex Optimization}.
\newblock In: Neural Information Processing Systems (NeurIPS) (2020)

\end{thebibliography}

\end{document}